\documentclass[11pt]{amsart}
\usepackage{hyperref,pb-diagram,amssymb,epic,eepic,verbatim,graphicx,graphics,epsfig,psfrag}
\hypersetup{colorlinks}
\usepackage{color}
\definecolor{darkred}{rgb}{0.5,0,0}
\definecolor{darkgreen}{rgb}{0,0.5,0}
\definecolor{darkblue}{rgb}{0,0,0.5}
\hypersetup{ colorlinks,
linkcolor=darkblue,
filecolor=darkgreen,
urlcolor=darkred,
citecolor=darkblue }
\usepackage{pb-diagram,amssymb,epic,eepic,verbatim,graphicx}
\usepackage{graphics,epsfig,psfrag,enumitem,braket} 
\newtheorem{theorem}{Theorem}[subsection]
\newtheorem{corollary}[theorem]{Corollary}

\newtheorem{proposition}[theorem]{Proposition}
\newtheorem{lemma}[theorem]{Lemma}
\newtheorem{lem}[theorem]{}
\theoremstyle{definition}
\newtheorem{definition}[theorem]{Definition}

\theoremstyle{remark}
\newtheorem{remark}[theorem]{Remark}
\newtheorem{example}[theorem]{Example}
\newcommand{\blem}{\begin{lem} \rm}
\newcommand{\elem}{\end{lem}}
\newcommand\A{\mathcal{A}}

\newcommand{\N}{\mathbb{N}}
\newcommand{\R}{\mathbb{R}}

\newcommand{\C}{\mathbb{C}}
\newcommand{\cC}{\mathcal{C}}

\newcommand{\Z}{\mathbb{Z}}

\newcommand{\ddt}{\frac{\d}{\d t}}

\renewcommand{\P}{\mathbb{P}}

\newcommand\lie[1]{\mathfrak{#1}}

\newcommand{\h}{\lie{h}}
\newcommand{\g}{\lie{g}}

\newcommand{\z}{\lie{z}}

\renewcommand{\u}{\lie{u}}

\newcommand{\su}{\lie{su}}

\newcommand{\on}{\operatorname}
\newcommand{\ainfty}{{$A_\infty$\ }}

\newcommand{\mgraph}{\on{graph}}

\newcommand{\dual}{\vee}

\newcommand{\Obj}{\on{Obj}}

\newcommand{\Symp}{\on{Symp}}

\newcommand{\Lag}{\on{Lag}}

\newcommand{\Sp}{\on{Sp}}

\newcommand{\Aut}{ \on{Aut} }

\newcommand{\Ad}{ \on{Ad} }

\newcommand{\Hom}{ \on{Hom}}

\renewcommand{\ker}{ \on{ker}}

\newcommand{\im}{ \on{im}}

\newcommand{\ssm}{\kern-.5ex \smallsetminus \kern-.5ex}

\newcommand\dirac{/\kern-1.2ex\partial} 
\newcommand\qu{/\kern-.7ex/} 
\newcommand\lqu{\backslash \kern-.7ex \backslash} 

\newcommand\dr{r_+ \kern-.7ex - \kern-.7ex r_-}
 
\renewcommand{\d}{{\operatorname{d}}}
\newcommand{\ol}{\overline}

\newcommand\eps{\epsilon}
\newcommand\Om{\Omega}

\newcommand{\lan}{\langle}
\newcommand{\ran}{\rangle}

\newcommand{\ti}{\tilde}

\newcommand\pt{\on{pt}}

\renewcommand{\ss}{\on{ss}}
\newcommand\Tr{\on{Tr}}

\newcommand\rank{\on{rank}}

\newcommand\Vect{\on{Vect}}
\newcommand\ul{\underline}

\newcommand\G{\mathcal{G}}

\newcommand\bdefn{\begin{definition}}
\newcommand\edefn{\end{definition}}
\newcommand\bea{\begin{eqnarray*}}
\newcommand\eea{\end{eqnarray*}}
\newcommand\bcv{\left[ \begin{array}{r} }
\newcommand\ecv{\end{array} \right] }

\newcommand\bma{\left[ \begin{array} }
\newcommand\ema{\end{array} \right]}
\newcommand\ben{\begin{enumerate}}
\newcommand\een{\end{enumerate}}
\newcommand\beq{\begin{equation}}
\newcommand\eeq{\end{equation}}
\newcommand\bex{\begin{example}}
\newcommand\bsj{\left\{ \begin{array}{rrr} }
\newcommand\esj{\end{array} \right\}}

\newcommand\Id{\on{Id}}

\newcommand\eex{\end{example}}
\newcommand\Crit{{\on{Crit}}}

\newcommand\sx{*\kern-.5ex_X}

\def\mathunderaccent#1{\let\theaccent#1\mathpalette\putaccentunder}
\def\putaccentunder#1#2{\oalign{$#1#2$\crcr\hidewidth \vbox
to.2ex{\hbox{$#1\theaccent{}$}\vss}\hidewidth}}

\renewcommand\sharp{\setlength{\unitlength}{0.00010333in}
\begin{picture}(688,703)(0,-10)
\path(244,644)(244,44)
\path(444,644)(444,44)
\path(44,444)(644,444)
\path(644,244)(44,244)
\end{picture}
}

\newcommand{\CC}{\mathcal{C}}

\newcommand\GFuk{{\on{Fuk}^{\sharp}}}
\newcommand\concat{{\sharp}}
\newcommand\Fuk{{\on{Fuk}}}

\newcommand{\Cob}{\on{Bor}}
\newcommand{\graph}{\on{graph}}

\begin{document}

%
%
%
%
%

\title[Floer field theory]{Floer field theory \\ for coprime rank and degree}

\author{Katrin Wehrheim} 

\address{Department of Mathematics,
Massachusetts Institute of Technology,
Cambridge, MA 02139.
{\em E-mail address: katrin@math.mit.edu}}

\author{Chris Woodward}

\address{Department of Mathematics, 
Rutgers University,
Piscataway, NJ 08854.
{\em E-mail address: ctw@math.rutgers.edu}}

\thanks{Partially supported by NSF grants CAREER 0844188 and DMS
  0904358}

\begin{abstract}
We construct partial category-valued field theories in
$2+1$-dimensions using Lagrangian Floer theory in moduli spaces of
central-curvature unitary connections with fixed determinant of rank
$r$ and degree $d$ where $r,d$ are coprime positive integers.  These
theories associate to a closed, connected, oriented surface the
Fukaya category of the moduli space, and to a connected
bordism between two surfaces a functor between extended Fukaya
categories.  We obtain the latter by combining Cerf theory with
holomorphic quilt invariants.  These functors satisfy the natural
composition law.
\end{abstract}  

\maketitle

\tableofcontents

\section{Introduction} 
\label{field}

Floer's instanton homology \cite{fl:inst} associates to any homology
three-sphere a $\Z/8\Z$-graded group that is a version of the Morse
homology of the Chern-Simons functional on the space of
$SU(2)$-connections.  This homology forms a natural receptacle for a
{\em relative instanton} invariant of four-manifolds with boundary,
defined by counting solutions to the anti-self-dual Yang-Mills
equations on the four-manifold obtained by attaching a cylindrical
end.  These invariants satisfy a gluing law for four-manifolds with a
common boundary component. In particular, the Donaldson invariants of
a four-manifold that splits along a homology three-sphere can be
computed from the relative invariants of the two parts
\cite{don:floer}.  The guiding principle of topological field theory
now asks for an invariant of three-manifolds with boundary that
satisfies a gluing law such that composition gives rise to the
instanton Floer homology for closed three-manifolds.  A natural
algebraic framework for such an invariant is that of category-valued
field theories: functor-valued invariants of bordisms satisfying a
composition law.  The most well-known example of such a theory is
associated with the {\em Wess-Zumino-Witten} model of conformal field
theory; it associates to a two-dimensional bordism a functor between
categories of representations of affine Lie algebras.  In particular,
the functor associated to the pair of pants gives these categories the
structure of a tensor product \cite{hu:ver}.  This theory extends to a
$1+1+1$-dimensional theory which associates to any closed
three-manifold the Witten-Reshetikhin-Turaev invariant.  The field
theory associated to the anti-self-dual Yang-Mills equations is hoped
to be a $2+1+1$-dimensional theory, including a category-valued
$2+1$-dimensional theory.  A strategy towards defining such an
invariant was proposed by Donaldson and Fukaya \cite{fuk:bo} who,
inspired by the Atiyah-Floer conjecture, suggested that the category
associated to the $2$-manifold should be a category of Lagrangian
submanifolds of the moduli space of flat $SU(2)$-bundles.

These moduli spaces, however, are singular due to unavoidable
reducible connections on the necessarily trivial bundles.  In order to
obtain smooth moduli spaces of flat connections, one can consider
nontrivial $SO(3)$-bundles.  Counting anti-self-dual connections on
such bundles give rise to instanton Floer homology for closed
$3$-manifolds with $b_1>0$.  These were some of the motivating
examples for the definition of the Fukaya category.  Fukaya also
proposed in \cite{fuk:bo} that a three-manifold with boundary should
give rise to an object of the dual of the Fukaya category. Such an
instanton Floer homology with Lagrangian boundary conditions was also
proposed by Salamon in \cite{Sa} and is rigorously constructed in
Salamon-Wehrheim \cite{SW}.  More generally, if we fix coprime
positive integers $r,d$, then there is a unique isomorphism class of
$U(r)$-bundles of degree $d$ on each surface $X$.  The moduli space of
central-curvature connections with fixed determinant is a smooth
symplectic manifold $M(X)$, described in more detail in Section
\ref{svft}.  To (various types of) symplectic manifolds one can
associate a Fukaya category or its cohomological category, the
Donaldson category whose objects are Lagrangian submanifolds and
morphism spaces are Floer homology groups.  We use this as starting
point for the construction of a $2+1$-dimensional connected
category-valued field theory.  The version of the Fukaya category
needed is an {\em extended Fukaya category} $\Fuk^{\sharp}(M(X))$ of
$M(X)$ whose objects are sequences of correspondences 
\begin{equation} \label{firstchange} \ul{L} = (L_{1},\ldots, L_{m}),
  \quad L_{j} \subset N_{j-1}^-\times N_{j}, \quad
  j=1,\ldots,m \end{equation}
  from a point $N_0 = \on{pt}$ to $N_j = M(X)$ as in
  \cite{quiltfloer}.  The composition maps use counts of {\em quilted
    disks} in Ma'u-Wehrheim-Woodward \cite{Ainfty}.

  In this paper we construct a functor from the category of (compact,
  connected, connected oriented $2$-manifolds, $3$-bordisms) to the
  category of (\ainfty categories, homotopy classes classes of \ainfty
  functors).  The theory associates to each compact oriented
  $2$-manifold $X$ the extended Fukaya category $\Fuk^{\sharp}(M(X))$
  of $M(X)$.  Theorem \ref{fft} below proves the existence of functors
$$\Phi(Y): \Fuk^{\sharp}(M(X_-))\to  \Fuk^{\sharp}(M(X_+))$$
between these Fukaya categories given by counting holomorphic quilts
with boundary and seams in spaces of central-curvature connections
over $3$-bordisms $Y$ from $X_-$ to $X_+$.  These functors satisfy a
natural composition law whenever a bordism is formed by gluing
bordisms $Y,Y'$ along a common boundary:
$$\Phi(Y \circ Y') = \Phi(Y')\circ\Phi(Y) .$$
Rather than using anti-self-duality equations, we apply a dimensional
reduction.  Thus we wish to associate to a bordism $Y$ a Lagrangian
correspondence
$$L(Y)\subset M(X_-)^- \times M(X_+)$$
arising from the moduli space of central-curvature fixed-determinant
connections on a bundle over $Y$.  While such moduli spaces on an
arbitrary bordism are not smooth, the moduli spaces for elementary
bordisms (cylinders or handle attachments) are smooth.  Thus a Cerf
decomposition of the bordism into elementary pieces provides a
sequence of Lagrangian correspondences.  A categorification functor
constructed from pseudoholomorphic quilts in \cite{Ainfty} can be used
to convert each Lagrangian correspondence to a functor.  To show
independence of the composition of these functors from the choice of
decomposition of the bordism we prove that Cerf moves between the
decompositions are reflected by an equivalence (embedded geometric
composition) between the associated Lagrangian correspondences, up to
a possible grading shift.  
\label{gradshift1}
Finally, equivalent sequences of correspondences give rise to
isomorphic functors by the ``strip-shrinking'' analysis developed in
\cite{ww:isom}.  Because we need to restrict to situations in which
there are no reducible connections to obtain smooth symplectic
manifolds, not all the axioms of a topological field theory are
satisfied.  For instance, the surfaces must be connected, which
precludes the product axiom.  Moreover, the invariant for closed
three-manifolds arising from this Floer field theory is trivial, since
the moduli spaces associated to two-spheres are empty.  However, via a
connect sum construction we associate to a closed three-manifold a
relatively $\Z/4\Z$-graded quilted Floer homology group (Definition
\ref{torussummed} below) which, by an unproven version of the
Atiyah-Floer conjecture, should agree with a $U(r)$ instanton Floer
homology of the connect sum.  Since the preprint was first circulated
other approaches and extensions have appeared in Manolescu-Woodward
\cite{mw} and Cazassus \cite{caz:symp}, \cite{caz:nat} who use Floer
theory in an extended moduli space, and Horton \cite{horton}, who uses
traceless character varieties.  \label{secondchange}

The structure of the paper is the following.  Precise definitions of
the involved categories can be found in Section~\ref{cerfsec}, leading
to a rigorous formulation of this construction framework.  This
strategy has already been applied to obtain various other gauge
theoretic $2+1$ field theories. For example, a sequel to this paper
\cite{fieldb} uses similar $U(r)$ moduli spaces to develop invariants
for tangles that conjecturally correspond to Floer homology invariants
arising from singular instantons developed by Kronheimer-Mrowka
\cite{km:kh}.  A first application of the results of this paper to
symplectic mapping class groups of representation varieties was given
by Smith \cite{sm:qu}.   \label{thirdchange}

The present paper is an updated and more detailed version of a paper
the authors have circulated since 2007. The authors have unreconciled
differences over the exposition in the paper, and explain their points
of view at
\href{https://math.berkeley.edu/~katrin/wwpapers/}{https://math.berkeley.edu/$\sim$katrin/wwpapers/}
resp.
\href{http://christwoodwardmath.blogspot.com/}{http://christwoodwardmath.blogspot.com/}. The
publication in the current form is the result of a mediation.

\section{Field theories via connected Cerf theory}
\label{cerfsec}

In this section we review a version of the theory of Cerf describing
the decomposition of connected bordisms into elementary bordisms
between connected surfaces, and the ``Cerf moves'' between different
decompositions. Then we explain how to build a field theory by
assigning morphisms to elementary bordisms so that certain ``Cerf
relations" are satisfied.  We begin by fixing the notation for field
theories.

\subsection{Field theories}
\label{fieldtheories}

Our language for topological field theories for bordisms adapts that
in, for example, Lurie \cite{lurie:class}, rephrasing the earlier
definition of Atiyah.  Let $n$ be a non-negative integer and $X_\pm$
compact oriented $n$-manifolds.  Since the theory of connected
bordisms is trivial for $n = 0,1$, we take $n \ge 2$.

\begin{definition} {\rm (Connected bordism category)} 
\begin{enumerate} 
\item 
A {\em bordism} from $X_-$ to $X_+$ is a compact, oriented
  $n+1$-manifold $Y$ with boundary equipped with an
  orientation-preserving diffeomorphism 
$$\phi: \partial Y \to \ol{X}_-
  \sqcup X_+ .$$  
Here $\ol{X}_-$ denotes the manifold $X_-$ equipped with the opposite
orientation.
\item The {\em connected bordism category} $\Cob_{n+1}^0$ is the
  category whose
\begin{enumerate} 
\item objects are compact, connected, oriented $n$-dimensional smooth
  manifolds $X$;
\item morphisms from $X_-$ to $X_+$ are $n+1$-dimensional connected
  bordisms $(Y,\phi)$ from $X_-$ to $X_+$ modulo the equivalence given
  as follows: Set two bordisms $(Y,\phi),(Y',\phi')$ from $X_-$ to
  $X_+$ to be equivalent if there exists an orientation-preserving
  diffeomorphism $\psi$ that extends the given diffeomorphism on their
  boundaries
$$\psi: Y \to Y', \quad \phi' \circ \psi |_{\partial Y} = \phi:
  \partial Y \to \ol{X}_- \sqcup X_+ ;$$
\item composition of morphisms is given by gluing bordisms together:
  Given two bordisms $(Y_1,\phi_1)$ from $X_0$ to $X_1$ and
  $(Y_2,\phi_2)$ from $X_1$ to $X_2$, we may glue them together to a
  bordism $(Y_1,\phi_1) \cup_{X_1} (Y_2,\phi_2)$ from $X_0$ to
  $X_2$. For an explicit construction choose collar neighborhoods
$$\kappa_1: (-\eps,0] \times X_1 \to Y_1, \quad \kappa_2: [0,\eps)
      \times X_1 \to Y_2$$ 
and define
$$Y_1 \circ Y_2 := Y_1 \sqcup Y_2 \sqcup ((-\eps,\eps) \times X_1)/
\sim$$
where $\sim$ is the obvious equivalence given by $\kappa_1,\kappa_2$.
The resulting bordism $Y_1 \circ Y_2$ is well-defined up to
equivalence, since any two choices of collar neighborhoods are
isotopic by \cite[Thm.1.4]{milnor:hcobord}.  Given two morphisms
$[Y_1] = [(Y_1,\phi_1)]$ and $[Y_2] = [(Y_2,\phi_2)]$ we define $[Y_1]
\circ [Y_2] = [Y_1 \circ Y_2]$ with the boundary identification
induced from $\phi_1$ and $\phi_2$, independent of the choice of
collar neighborhood.
\item The identity morphism for a manifold $X$ is represented by the
  trivial bordism $[0,1] \times X$ with the obvious identifications on
  the boundary.
\end{enumerate}
\end{enumerate}
\end{definition}

\begin{definition} \label{diffeo comp} {\rm (Connected field theories)}  
Let $\cC$ be a category.  A {\em $\cC$-valued connected field theory}
in $n+1$ dimensions is a functor $\Phi$ from $\Cob_{n+1}^0$ to $ \cC$.
\end{definition} 

For example, a {\em connected \ainfty-category-valued field theory} is
a field theory taking values in the category of (\ainfty categories,
homotopy classes of \ainfty functors).  In Section~\ref{symp} below we
will construct a connected \ainfty-category-valued field theory by
composing a symplectic-valued field theory and the categorification
functor from \cite{Ainfty}, which associates \ainfty functors to
Lagrangian correspondences.

Field theories usually allow for disconnected manifolds and
bordisms. In this case one would take $\cC$ to be a symmetric monoidal
category and require the product axiom $\Phi(X_0 \sqcup X_1) =
\Phi(X_0) \otimes \Phi(X_1)$ for disjoint unions $X_0 \sqcup
X_1$. However, in our examples the Fukaya categories associated to
$2$-manifolds are well defined only in the connected case: Otherwise
the underlying symplectic space, a moduli space of bundles over a
disconnected surface, may be singular.  Hence we have restricted to
the connected bordism category.

\subsection{Cerf decompositions of bordisms}

In this section we describe the decomposition of bordisms into
elementary pieces in the connected bordism category.  In the
following, let $X_-, X_+$ be compact, connected, oriented manifolds of
dimension $n\geq 1$, and let $(Y, \phi: \partial Y \to \ol{X}_- \sqcup
X_+)$ be a compact, connected, oriented bordism from $X_-$ to $X_+$.

\begin{definition}   \label{Morse datum}
{\rm (Elementary and cylindrical bordisms)} 
\begin{enumerate} 
\item A {\em Morse datum} for $(Y,\phi)$ consists of a pair $(f,\ul{b})$ of a
  Morse function $f:Y\to\R$ and an ordered tuple 
\begin{equation} \label{ulb} 
\ul{b}= ( b_0 < b_1  < \ldots < b_{m} ) \in \R^{m+1} \end{equation}
such that
\begin{enumerate}
\renewcommand{\labelenumi}{(\roman{enumi})}
\item \label{firstcond} the minima and maxima of $f$ are 
$$\phi^{-1}(X_-)=f^{-1}(b_0), \quad \phi^{-1}(X_+)=f^{-1}(b_m) ;$$
\item \label{conn} each level set $f^{-1}(b)$ for $b\in\R$ is
  connected;
\item $f$ has distinct values at the (isolated) critical points,
  i.e.\ it induces a bijection $\Crit f\to f(\Crit f)$ between
  critical points and critical values;
\item \label{fourth}
$b_0,\ldots, b_{m} \in \R \setminus f(\Crit f)$ are regular values of
  $f$ such that each interval $(b_{i-1},b_i), i = 1,\ldots,m$ contains
  at most one critical value of $f$: 
$$ \# \Crit(f) \cap f^{-1}(b_{i-1},b_i)\leq 1 .$$
\end{enumerate}
A Morse function $f: Y \to \R$ is {\em adapted} to $Y$ if the first
condition above \eqref{firstcond} holds.
\item 
 \label{elementary}
A connected bordism $(Y,\phi)$ is an {\em elementary bordism} if $Y$
admits a Morse datum $(f,\ul{b}=(\min f,\max f))$, that is $f$ is a
Morse function with at most one critical point.  
\item $(Y,\phi)$ is a {\em cylindrical bordism} if $Y$ admits a Morse
  datum $(f,\ul{b}=(\min f,\max f))$, where $f$ is a Morse function
  with no critical point.
\item 
A morphism $[(Y,\phi)]$ of $\Cob^0_{n+1}$ is an {\em elementary
  resp.\ cylindrical morphism} if one (and hence all) of its
representatives is an elementary resp.\ cylindrical bordism.
\end{enumerate} 
\end{definition}

\begin{definition} \label{cerfdecomp} {\rm (Cerf decompositions)}  
\begin{enumerate}
\item 
A {\em Cerf decomposition} of the bordism $(Y,\phi)$ is a decomposition
$$
Y = Y_1 \cup_{X_1} Y_2 \cup_{X_2}  \ldots \cup_{X_{m-1}} Y_{m} 
$$ into a sequence $( Y_i \subset Y )_{i=1,\ldots m}$ of elementary
bordisms embedded in $Y$ that are disjoint from each other and
$\partial Y$ except for 
$$Y_1\cap \partial Y = \phi^{-1}(X_-) , \quad 
Y_m\cap \partial Y = \phi^{-1}(X_+), \quad X_i:=Y_{i}\cap Y_{i+1} $$ 
that are also connected submanifolds in $Y$ of codimension $1$.  As a
consequence we have
$$\partial Y_i \cong \ol{X}_{i-1}\sqcup X_i, \ \ i=1,\ldots,m, \quad
X_0=\phi^{-1}(X_-), \quad X_m=\phi^{-1}(X_+) .$$
\item
A {\em Cerf decomposition} of the morphism $[Y]$ in the connected
bordism category $\Cob_{n+1}$ is a sequence $ ([Y_i])_{i=1,\ldots m}$
of elementary morphisms that compose to
$$ [(Y,\phi)]= [(Y_1,\phi_1)] \circ [(Y_2,\phi_2)] \circ \ldots \circ
[(Y_{m},\phi_m)] .
$$
\item Two Cerf decompositions
$$ [(Y,\phi)] = [(Y_1,\phi_1)] \circ \ldots \circ
[(Y_m,\phi_m)] = [(Y_1,\phi_1)'] \circ \ldots \circ
[(Y_m,\phi_m)'] $$
are {\em equivalent} if there exist orientation-preserving
diffeomorphisms
$$ \delta_0 = \on{Id}_{X_0}, \quad \delta_1: X_1 \to X_1',\ldots,
\delta_{m-1}: X_{m-1} \to X_{m-1}', \quad \delta_m = \on{Id}_{X_m}
$$
such that for each $i = 1,\ldots, m$,
$$ [(Y_i,\phi_i)] = [(Y_i', ( \delta_{i-1} \sqcup \delta_i) \circ
  \phi_i )] .$$
\end{enumerate}
\end{definition}

\begin{remark} \label{lem cerf is morse} {\rm (Cerf decomposition via Morse datum)}  
Any Morse datum $(f,\ul{b})$ for the bordism $(Y,\phi)$ induces a Cerf
decomposition 
$$Y = Y_1 \cup_{X_1} \ldots \cup_{X_{m-1}} Y_{m}, \quad ( Y_i :=
f^{-1}( [b_{i-1},b_i] ) )_{i=1,\ldots m}$$
into elementary bordisms between the connected level sets $X_i =
f^{-1}(b_i)$.  Moreover, any Cerf decomposition of a representative
$Y\in[Y]$ induces a Cerf decomposition of the morphism $[Y]$.  On the
other hand, any Cerf decomposition of a bordism or morphism arises
from a Morse datum.
\end{remark}

\begin{remark} \label{handles} {\rm (Handle attachments)}  
The pieces $Y_i$ of a Cerf decomposition have simple topological
descriptions as follows: If the elementary bordism $Y_i$ contains no
critical point then it is in fact a cylindrical bordism.  In that
case $Y_i$ is diffeomorphic to the {\em cylinder} $X_i \times
[b_{i-1},b_i]$, and $X_{i-1}$ is diffeomorphic to $X_i$.  

Suppose $Y_i$ contains a single critical point with index
$k\in\{1,\ldots,n\}$.  In that case $Y_i$ is obtained (up to homotopy)
from the incoming manifold $X_{i-1}$ by attaching a handle $B^{k}
\times B^{n - k}$.  Here $B^k$ is a $k$-ball and the handle $B^k
\times B^{n-k}$ is attached via an {\em attaching map} $\alpha:
S^{k-1} \times B^{n - k} \to X_{i-1}$:
$$ Y_i \cong X_{i-1} \sqcup (B^k \times B^{n-k}) / ( x \sim \alpha(x),
x \in S^{k-1} \times B^{n-k}) .$$
The image of $S^{k-1} \times \{ 0 \}$ in $X_{i-1}$ is an {\em
  attaching cycle} and the image of $S^{k-1} \times B^{n-k}$ is viewed
as a thickening of the attaching cycle; we often omit the thickenings
from the description.

Conversely, $Y_i$ can be obtained from the outgoing manifold $X_i$ up
to homotopy by attaching a handle of opposite index to an attaching
cycle in $X_i$.  Concrete attaching cycles can be specified by
choosing a metric on $Y_i$.  Then the attaching cycles in $X_{i-1}$
resp.\ $X_i$ are given by the intersection with the stable resp.\
unstable manifold for the upward gradient flow of $f$ from the unique
critical point in $Y_i$, see Figure~\ref{stable}.  As explained in
Milnor \cite{milnor:hcobord}, the notion of attaching cycle can also
be defined without a metric via a {\em gradient-like vector field},
that is, a vector field $v \in \Vect(Y_i)$ \label{fourthchange} such that $D_v f > 0$
everywhere except at the critical points. This notion is sometimes
useful in order to, for example, show that any Cerf decomposition can
be re-arranged so that the indices of the critical points are in
increasing order, see \cite[Theorem 8.1]{milnor:hcobord}.
\begin{figure}[h]
\includegraphics[width=3in,height=1.5in]{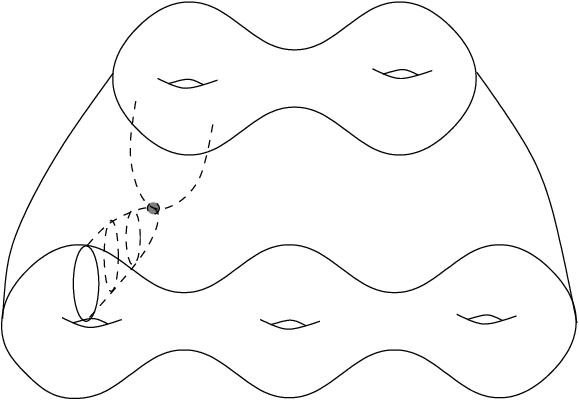}
\caption{Stable and unstable manifolds inside an elementary bordism} 
\label{stable}
\end{figure}
\end{remark}

\begin{lemma}   
\label{existence}  For $n \ge 2$ any bordism $Y$ as above admits a Morse datum.
\end{lemma}

\begin{proof}  
By Milnor \cite[Theorem 8.1]{milnor:hcobord}, there exists an adapted
Morse function $f: Y \to \R$ such that $f$ is {\em self-indexing} in
the sense that the critical points of index $i$ have critical value
$i$, for each $i =0,\ldots, n+1$, and furthermore there are no
critical points of index $0$ or $n+1$.  After a small perturbation, we
may assume that the critical values of $f$ are distinct, by Milnor
\cite[Chapter 4]{milnor:hcobord}, but still with the property that 
the order on critical values is the same as that on index:
\begin{equation} \label{order} \forall y,y' \in \Crit(f), \ \ (i(y) < i(y')) \implies (f(y) <
f(y)') .\end{equation} 
This ordering property \eqref{order} implies that the fibers of $f$
are connected.  Indeed, each level set $f^{-1}(b)$ is obtained by
attaching handles to lower level sets $f^{-1}(b'), b' < b$; the level
sets $f^{-1}(b)$ become disconnected either because of a critical
point $y \in \Crit(f)$ of index $0$, which does not exist by
assumption, or by attaching a handle of index $n$ with disconnecting
attaching cycle.  Once a level set $f^{-1}(b)$ is disconnected, it can
be become connected again only by attaching a handle of index one,
with the points of the attaching cycle in different components of
$f^{-1}(b)$.  But since the Morse function $f$ is self-indexing and $n
\ge 2$, the $n$-handles are attached after the $1$-handles.  The
existence of a disconnecting $n$-handle would imply that $X_+$ is
disconnected, a contradiction.  Given such a Morse function $f$, let
$$b_0:=\min f, \quad b_m:=\max f, \quad m\geq\#\Crit f. $$
There evidently always exists a choice of $b_1<\ldots<b_{m-1}$
satisfying condition \eqref{fourth}, hence making $(f,\ul{b})$ a Morse
datum.
\end{proof} 

Note that our definition of a Cerf decomposition differs from the
standard handle decomposition in that we allow the elementary bordisms
$Y_i$ to be cylindrical bordisms and we do not keep track of the
attaching cycles.  This definition also simplifies the moves between
different decompositions: Since we do not fix a metric or require the
Smale condition of stable and unstable manifolds intersecting
transversally, we need not consider the handle slide move discussed by
Kirby \cite[p. 40]{kir:calc}.  On the other hand, we use much finer
decompositions than Heegaard splittings which are commonly used to
define topological invariants via Floer theory.  The latter are
specific to dimension $3$ and for bordisms defined as follows.

\begin{definition}   \label{compbody} {\rm (Heegaard splittings of bordisms)}  Let $Y$ be a compact
connected oriented bordism between compact connected oriented surfaces
$X_\pm$.
\begin{enumerate}
\item
$Y$ is a {\em compression body} if it admits a Morse function such that
  all critical points have the same index (namely $1$ or $2$):
$$ \exists f: Y \to \R \ \text{Morse}, \quad ( y_1,y_2 \in \Crit(f))
  \ \implies \ ( i(y_1) = i(y_2)) .$$
Equivalently, $Y$ is obtained from $\phi^{-1}(X_-)$ by adding only
handles of the same index, or from $\phi^{-1}(X_+)$ by adding only
handles of the opposite index.
\item
A {\em Heegaard splitting} of $Y$ is a decomposition $Y=Y_-\cup_X Y_+$
into compression bodies $Y_-,Y_+$ with a common boundary $X$, such
that $Y_-$ contains $\phi^{-1}(X_-)$ and $Y_+$ contains $\phi^{-1}(X_+)$, and both are
obtained from these boundary components by adding handles of index
$1$. We call $X$ the {\em Heegaard surface} of the splitting.
\end{enumerate}
\end{definition}  
 
\begin{remark}  {\rm (Heegaard splittings via Morse functions)} 
Let $f:Y\to\R$ be an adapted Morse function such that all critical
points of index $1$ have values less than those of the critical points
of index $2$.  Pick a value $c \in \R$ that separates the critical
values of index $1$ from those of index $2$:
$$ \forall y_1,y_2 \in \Crit(f), \quad (f(y_1) < f(y_2)) \ \iff \ (i(y_1) < i(y_2)) .$$ 
Then 
$$Y = Y_- \cup Y_+, \quad Y_- = f^{-1} (-\infty,c], \quad Y_+ = f^{-1}[c,\infty)$$ 
form a Heegaard splitting of $Y$.  Note that any such $f$ also
satisfies (ii) in Definition~\ref{Morse datum} automatically, and can
be perturbed to satisfy (iii).  The function $f$ can then be completed
to a Morse datum $(f,\ul{b})$ that induces a special Cerf
decomposition of $Y$, whose elementary bordisms are ordered by index.
\end{remark}  

\begin{remark} \label{eq is diffeq} Given two representatives $Y$ and
  $Y'$ of the same morphism in $\Cob^0_{n+1}$ and a diffeomorphism
  $\psi:Y\to Y'$, any Morse datum $(f,\ul{b})$ for $Y'$ induces a
  Morse datum $(\psi^*f,\ul{b})$ for $Y$ by pullback.  The Cerf
  decompositions of $Y$ induced by $(f,\ul{b})$ and $(\psi^*f,\ul{b})$
  are then equivalent via the collection
  $(\psi|_{Y_i})_{i=1,\ldots,m}$ of diffeomorphisms of the elementary
  bordisms.
\end{remark}

The existence of Morse data in Lemma~\ref{existence} implies that
every morphism in the bordism category has a Cerf decomposition.  The
subsequent Cerf Theorem~\ref{cerf} implies that Cerf decompositions
are unique up to the Cerf moves, which will be defined in the
following.  For simplicity of notation, we drop the boundary
identifications from the notation.

\begin{definition} {\rm (Cerf moves)}  Let $Y$ be a bordism and  $[Y]=[Y_1]\circ\ldots\circ[Y_m]$
a Cerf decomposition.  A {\em Cerf move} is one of
the following operations on $([Y_i])_{i=1,\ldots, m}$.
\begin{enumerate} 
\item 
A {\em critical point cancellation} is the move
$$
\;\qquad \text{from} \quad
[Y] = \ldots  [Y_j] \circ [Y_{j+1}]  \ldots 
\quad
\text{to}
\quad
[Y] = \ldots [ Y_j \cup Y_{j+1} ]  \ldots ,
$$ 
where for some $j\in\{1,\ldots,m-1\}$ the two consecutive elementary
bordisms $Y_j$, $Y_{j+1}$ compose to a cylindrical bordism $Y_j
\cup Y_{j+1} \subset Y$.  More precisely, in this situation, critical
point cancellation is the move from $([Y_i])_{i=1,\ldots,m}$ to
$([Y'_i])_{i=1,\ldots,m'}$ with $m'=m-1$ given by 
$$[Y'_i]=[Y_i], \ i<j, \quad  [Y'_j]=[Y_j\cup Y_{j+1}], \quad [Y'_i]=[Y_{i+1}], \ i>j .$$  
A {\em critical point creation} is the same move with the roles of
$([Y_i])_{i=1,\ldots,m}$ and $([Y'_i])_{i=1,\ldots,m'}$ interchanged.
\item 
A {\em critical point switch} is the move
$$
\;\qquad\text{from} \quad
[Y] = \ldots [Y_j] \circ [Y_{j+1}]  \ldots 
\quad
\text{to}
\quad
[Y]= \ldots   [Y'_j] \circ [Y'_{j+1}]  \ldots ,
$$ 
where for some $j\in\{1,\ldots,m-1\}$ the composition $[Y_j] \circ
[Y_{j+1}]$ equals $[Y'_j] \circ [Y'_{j+1}]$, and the two Cerf
decompositions $[Y_j] \circ [ Y_{j+1}] = [ Y'_j ] \circ [ Y'_{j+1}]$
of the same morphism are given by Morse data $(f, \ul{b})$ and $(f',
\ul{b}')$ on representatives with unique critical points $y_{j(+1)}\in Y_{j(+1)}$ and
$y'_{j(+1)}\in Y'_{j(+1)}$ in each part, whose attaching cycles (for
some choice of a metric) switch in the following sense: The attaching
cycles of $y_j$ and $y_{j+1}$ in $X_j$ and those of $y'_j$ and
$y'_{j+1}$ in $X'_j$ are disjoint, while in $X_{j-1}=X'_{j-1}$ the
attaching cycle of $y_j$ is homotopic to that of $y'_{j+1}$, and the
attaching cycle of $y_{j+1}$ is homotopic to that of $y'_{j}$; and
analogously for the intersections of stable manifolds with
$X_{j+1}=X'_{j+1}$.  More precisely, in this situation, critical point
switch is the move from $([Y_i])_{i=1,\ldots,m}$ to
$([Y'_i])_{i=1,\ldots,m'}$ with $m'=m$, $[Y'_i]=[Y_i]$ for $i<j$ and $i>j+1$, and
$[Y'_j]\circ [Y'_{j+1}]=[Y_j]\circ [Y_{j+1}]$ as above.
\item
A {\em cylinder cancellation} is the move
$$
\;\qquad \text{from} \quad
[Y] = \ldots [Y_j] \circ [Y_{j+1}] \ldots 
\quad
\text{to}
\quad
[Y] = \ldots [ Y_j\cup Y_{j+1} ]  \ldots ,
$$ where for some $j\in\{1,\ldots,m-1\}$ one of the two consecutive
elementary morphisms $[Y_j]$, $[Y_{j+1}]$ is cylindrical.  Then the composition
$[Y_j] \circ [Y_{j+1}]$ is an elementary morphism as well.  More
precisely, in this situation, critical point cancellation is the move
from $([Y_i])_{i=1,\ldots,m}$ to $([Y'_i])_{i=1,\ldots,m'}$ with $m'=m-1$,
$$[Y'_i]=[Y_i], \ i<j, \quad [Y'_j]=[Y_j] \circ [Y_{j+1}], \quad [Y'_i]=[Y_{i+1}]$$ 
for $i>j$.  A {\em cylinder creation} is the same move with the roles
of $([Y_i])_{i=1,\ldots,m}$ and $([Y'_i])_{i=1,\ldots,m'}$
interchanged.
\end{enumerate}
\end{definition}

\begin{remark}  {\rm (Stabilizations of Heegaard splittings versus Cerf moves)} 
In dimension $n=2$ we can compare critical point creation to the
stabilization of Heegaard splittings.  A {\em stabilization} of a
Heegaard splitting is obtained by connect sum with the standard
Heegaard splitting of a sphere $S^3 = H_- \cup H_+$ into solid tori
$H_-,H_+$.  More precisely, given a Heegaard splitting $Y = Y_- \cup_X
Y_+$, its stabilization is obtained by pulling the Heegaard splitting
$Y {\sharp} S^3 = Y_-' \cup Y_+'$ into compression bodies $Y_\pm' =
Y_\pm {\sharp} H_\pm$ back to $Y\cong Y {\sharp} S^3$.  Equivalently,
let $[-1,1 ] \times X = Y_-'' \cup Y_+''$ be the decomposition of the
cylindrical bordism consisting of two elementary bordisms, each
carrying a Morse function with index $1$ resp. $2$.  Then $Y_\pm'$ is
obtained from $Y_\pm$ by attaching $Y_\pm''$ at $X$, since attaching a
one-handle is equivalent to connected sum with a torus.  Thus if the
Heegaard splitting $Y = Y_- \cup Y_+$ is induced by a Cerf
decomposition then the stabilization is obtained from a critical point
creation.  Conversely any critical point creation can be viewed as a
connected sum as above, and so induces a stabilization of the
corresponding Heegaard splitting.
\end{remark} 
    
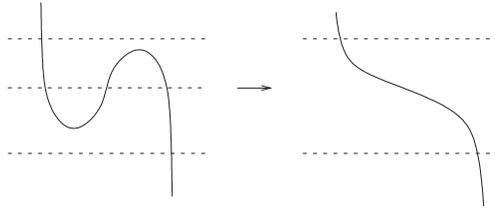
\begin{figure}[h]
\setlength{\unitlength}{0.00037489in}
\begingroup\makeatletter\ifx\SetFigFont\undefined
\def\x#1#2#3#4#5#6#7\relax{\def\x{#1#2#3#4#5#6}}%
\expandafter\x\fmtname xxxxxx\relax \def\y{splain}%
\ifx\x\y   
\gdef\SetFigFont#1#2#3{%
  \ifnum #1<17\tiny\else \ifnum #1<20\small\else
  \ifnum #1<24\normalsize\else \ifnum #1<29\large\else
  \ifnum #1<34\Large\else \ifnum #1<41\LARGE\else
     \huge\fi\fi\fi\fi\fi\fi
  \csname #3\endcsname}%
\else
\gdef\SetFigFont#1#2#3{\begingroup
  \count@#1\relax \ifnum 25<\count@\count@25\fi
  \def\x{\endgroup\@setsize\SetFigFont{#2pt}}%
  \expandafter\x
    \csname \romannumeral\the\count@ pt\expandafter\endcsname
    \csname @\romannumeral\the\count@ pt\endcsname
  \csname #3\endcsname}%
\fi
\fi\endgroup
{\renewcommand{\dashlinestretch}{30}
\begin{picture}(6774,2829)(0,-10)
\path(4512,2667)(4512,2665)(4513,2659)
	(4514,2650)(4515,2635)(4517,2614)
	(4520,2588)(4524,2556)(4529,2520)
	(4536,2480)(4543,2438)(4551,2395)
	(4560,2351)(4570,2307)(4581,2265)
	(4593,2225)(4607,2186)(4621,2150)
	(4636,2116)(4653,2084)(4671,2054)
	(4691,2026)(4712,2000)(4735,1975)
	(4761,1952)(4788,1929)(4818,1908)
	(4850,1887)(4875,1872)(4901,1857)
	(4929,1842)(4957,1828)(4988,1813)
	(5019,1799)(5052,1784)(5086,1769)
	(5122,1754)(5158,1739)(5196,1724)
	(5235,1709)(5274,1694)(5315,1678)
	(5356,1662)(5397,1646)(5440,1630)
	(5482,1614)(5525,1597)(5567,1580)
	(5609,1563)(5652,1546)(5693,1529)
	(5734,1511)(5775,1493)(5814,1475)
	(5853,1457)(5891,1439)(5927,1421)
	(5963,1402)(5997,1383)(6030,1364)
	(6061,1344)(6092,1325)(6120,1305)
	(6148,1284)(6174,1263)(6200,1242)
	(6226,1217)(6251,1192)(6274,1165)
	(6296,1137)(6317,1108)(6336,1077)
	(6353,1044)(6370,1010)(6385,973)
	(6400,934)(6414,892)(6426,848)
	(6438,802)(6449,753)(6460,701)
	(6470,647)(6479,592)(6487,534)
	(6495,477)(6502,419)(6509,361)
	(6514,306)(6519,253)(6524,205)
	(6527,161)(6530,122)(6532,89)
	(6534,63)(6535,43)(6536,28)
	(6537,19)(6537,14)(6537,12)
\dashline{60.000}(4062,2307)(6762,2307)
\dashline{60.000}(4062,732)(6762,732)
\path(3162,1632)(3612,1632)
\path(3492.000,1602.000)(3612.000,1632.000)(3492.000,1662.000)
\path(462,2802)(462,2800)(462,2795)
	(462,2786)(462,2771)(463,2751)
	(463,2725)(464,2692)(464,2653)
	(465,2609)(466,2560)(467,2507)
	(468,2451)(470,2394)(471,2336)
	(473,2277)(475,2220)(477,2164)
	(479,2110)(481,2058)(484,2008)
	(487,1961)(489,1916)(492,1874)
	(496,1834)(499,1797)(503,1761)
	(507,1728)(511,1696)(516,1666)
	(520,1637)(526,1610)(531,1583)
	(537,1557)(545,1524)(554,1492)
	(564,1462)(574,1432)(585,1402)
	(597,1374)(609,1346)(622,1320)
	(636,1294)(651,1269)(666,1246)
	(681,1224)(697,1203)(714,1184)
	(730,1166)(747,1150)(764,1136)
	(781,1123)(798,1112)(815,1102)
	(831,1094)(848,1088)(864,1083)
	(880,1080)(896,1078)(912,1077)
	(929,1078)(947,1080)(964,1084)
	(982,1089)(1000,1096)(1018,1104)
	(1036,1114)(1054,1125)(1073,1137)
	(1091,1151)(1109,1166)(1127,1181)
	(1144,1198)(1161,1215)(1177,1233)
	(1193,1252)(1207,1270)(1221,1289)
	(1234,1308)(1246,1326)(1258,1345)
	(1268,1363)(1278,1381)(1287,1399)
	(1297,1421)(1306,1443)(1314,1465)
	(1322,1488)(1329,1510)(1336,1534)
	(1343,1558)(1349,1582)(1356,1606)
	(1362,1631)(1368,1655)(1375,1680)
	(1381,1704)(1388,1727)(1395,1750)
	(1402,1772)(1410,1794)(1418,1815)
	(1427,1836)(1437,1857)(1447,1876)
	(1458,1895)(1469,1914)(1482,1933)
	(1496,1952)(1510,1971)(1526,1990)
	(1543,2009)(1560,2027)(1578,2044)
	(1597,2061)(1617,2077)(1637,2092)
	(1656,2105)(1676,2117)(1696,2128)
	(1716,2137)(1736,2144)(1755,2150)
	(1774,2154)(1793,2156)(1812,2157)
	(1827,2156)(1842,2155)(1856,2152)
	(1872,2148)(1887,2143)(1902,2137)
	(1918,2129)(1933,2119)(1949,2109)
	(1965,2096)(1980,2083)(1996,2067)
	(2011,2050)(2027,2032)(2042,2012)
	(2056,1991)(2070,1969)(2084,1945)
	(2097,1920)(2109,1894)(2121,1866)
	(2132,1838)(2143,1809)(2153,1779)
	(2162,1747)(2171,1715)(2179,1682)
	(2187,1647)(2192,1621)(2197,1593)
	(2202,1565)(2206,1535)(2211,1504)
	(2215,1472)(2218,1438)(2222,1402)
	(2225,1364)(2228,1324)(2231,1282)
	(2234,1237)(2236,1190)(2239,1140)
	(2241,1088)(2243,1033)(2245,976)
	(2247,917)(2249,856)(2251,794)
	(2252,730)(2254,667)(2255,604)
	(2256,542)(2257,483)(2258,426)
	(2259,374)(2260,326)(2260,284)
	(2261,247)(2261,217)(2262,193)
	(2262,174)(2262,162)(2262,153)
	(2262,149)(2262,147)
\dashline{60.000}(12,2307)(2712,2307)
\dashline{60.000}(12,1632)(2667,1632)
\dashline{60.000}(12,732)(2712,732)
\end{picture}
}
\caption{Critical point cancellation}
\label{minmax1}
\end{figure}

\begin{theorem}[Connected Cerf theory] \label{cerf} 
Let $Y$ be a bordism as before of dimension $n + 1 \ge 3$, and fix a
Cerf decomposition $[Y]=[Y_1]\circ\ldots\circ[Y_m]$. Then any other
Cerf decomposition of $[Y]$ can be obtained up to equivalence from
$([Y_i])_{i=1,\ldots,m}$ by a finite sequence of Cerf moves.
\end{theorem}

\begin{proof}[Remarks on the Proof of Theorem~\ref{cerf}:]
The statement without the connectedness conditions can be proved using
theorems of Thom and Mather \cite{ma:st5}, see also \cite{mar:sing}.
A generic homotopy $\ti{f}: Y \times [0,1] \to \R$ between two Morse
functions $f_0,f_1$ with distinct critical values has only a finite
number of cusp singularities, corresponding to the critical point
cancellations and creations, and a finite number of times $s \in
[0,1]$ such that the critical values of $\ti{f}_s := \ti{f}( \cdot,
s)$ are not distinct, and at such times two critical values cross.  A
homotopy between Morse functions with connected fibers does not
necessarily have connected fibers.  However, Gay and Kirby
\cite[Theorem 2]{gk:indef} show that there exists a generic homotopy
between any two Morse functions with connected fibers.  Let
$c_1,\ldots,c_m \in (0,1)$ be the times for which either the critical
values of $\ti{f}_s$ coincide or $\ti{f}_s$ is not Morse.

Away from the critical values the Cerf decompositions are equivalent
by diffeomorphisms.  Indeed, choose $\eps$ small and smoothly varying
$b_1(s),\ldots, b_{m-1}(s)$ separating the critical values of
$\ti{f}_s $ for $s \in [c_i + \eps, c_{i+1}-\eps]$.  The inverse
images of the level sets $\ti{f}_s^{-1}(b_i(s))$ flow out smooth
submanifolds of $Y \times [c_i + \eps,
c_{i+1}-\eps]$
\label{fifthchange} denoted $\ti{f}^{-1}(b_i)$, by the implicit
function theorem.  Choose a vector field
$v \in \Vect( Y \times [c_i + \eps, c_{i+1}-\eps])$ tangent to the
level sets $\ti{f}^{-1}(b_i)$ and satisfying
$(D_{y,s}\pi_2)_* v =\partial_s$ for any
$(y,s)\in Y \times [c_i + \eps, c_{i+1}-\eps]$, where $\pi_2$ is
projection onto the second factor.  Such a vector field $v_0$ exists
on each level set $\ti{f}^{-1}(b_i)$ since the $b_i(s)$ are regular
values:
\begin{equation} \label{Tys} 
T_{y,s} \ti{f}^{-1}(b_i) \cap (T_y Y \times \{ 0 \}) = T_y
f^{-1}(b_i(s)), \quad  D\pi_2 ( T_{y,s} \ti{f}^{-1}(b_i) ) = \R
.\end{equation} 
Next $v_0$ extends to a neighborhood of each level set
$\ti{f}^{-1}(b_i)$ by the tubular neighborhood theorem.  One may then
extend $v_0$ to a vector field on $Y \times [c_i + \eps,
  c_{i+1}-\eps])$ using interpolation with the vector field
$\partial_s \in \Vect( Y \times [c_i + \eps, c_{i+1}-\eps])$. That is,
let $\rho \in C^\infty(Y \times [c_i + \eps, c_{i+1} -
\eps])$ \label{change6} be a
bump function equal to one on a neighborhood of each
$\ti{f}^{-1}(b_i)$ and vanishing outside of a small neighborhood of
the union of $\ti{f}^{-1}(b_i)$.  Define $ v = \rho v_0 + (1 - \rho)
\partial_s $.  The flow $\psi_s$ of $v$ preserves the level sets of
$b_0,\ldots, b_m$ and so defines diffeomorphisms of the pieces of the
Cerf decomposition of $Y$ for $\ti{f}_s$:
$$ \psi_{s_2 - s_1} ( \ti{f}^{-1}_{s_1}(b_i(s_1), b_{i+1}(s_1)) ) = 
  \ti{f}^{-1}_{s_2}(b_i(s_2), b_{i+1}(s_2))   .$$
Hence the functions $\ti{f}_s$ and values $b_i(s)$ for $s \in [c_i +
  \eps, c_{i+1} - \eps]$ define an equivalent Cerf decomposition of
$[Y]$.

On the other hand, the Cerf decompositions for $c_i - \eps, c_i +
\eps$ are equal for all but one or two pieces by the same argument in
the previous paragraph.  For those pieces, one either has a critical
point switch move or critical point cancellation by the local model
for the cusp singularities \cite[p.157]{ma:st5}.
 \end{proof} 

\begin{remark} {\rm (Alternative approaches in dimension three)}  
Alternatively in dimension three, one may show that given a sequence
of Cerf moves with possibly disconnected fibers, one may modify the
sequence so that one obtains a sequence of Cerf moves preserving
connectedness, but so that the sequence is not necessarily associated
to a homotopy.  This is the approach taken in Juhasz
\cite[p. 1434-1437]{juhasz:hol}\footnote{We thank T. Perutz for
  pointing out this reference to us.}.  Finally, one may reduce the
theorem to a relative case of the Reidemeister-Singer theorem (that
any two Heegaard splittings are related by a sequence of
stabilizations and de-stabilizations).  By Motto's thesis,
\cite[Chapter 4]{motto:thesis} any two Heegaard splittings of a
connected bordism $Y$ are related by a sequence of stabilizations and
de-stabilizations.  In order to deduce Theorem \ref{cerf}, it suffices
to modify the given Cerf decomposition to one corresponding to a
Heegaard splitting by a sequence of critical point switches.  That
this is possible follows from Milnor \cite[Theorem 4.4,4.1,4.2
  Extension]{milnor:hcobord}.
\end{remark}

The Cerf theorem above implies that for any category $\cC$, in order
to construct a connected $\cC$-valued field theory in the sense of
Definition~\ref{diffeo comp}, it suffices to construct the functors on
elementary bordisms and check that the Cerf moves correspond to
composition identities in $\CC$.  For objects $M_1,M_2,M_3$ of $\cC$
we denote by $\circ: \Hom(M_1,M_2) \times \Hom(M_2,M_3) \to
\Hom(M_1,M_3)$ the composition map.  For any object $M$ of $\cC$
denote by $1_M \in \Hom(M,M)$ the identity.

\begin{theorem} \label{cerfthm} {\rm (Field theories via morphisms for
    elementary bordisms)} Suppose we are given for some $n \ge
  2$
  \label{change7} a partial functor $\Phi$ from $\Cob^0_{n+1}$ to
  $\CC$ that associates
\begin{enumerate}
\item 
to each compact oriented $n$-manifold $X$, an object 
$\Phi(X) \in \Obj(\cC)$,
\item 
to each elementary morphism $[Y]$ from $X_-$ to $X_+$, a morphism
$\Phi([Y])$ from $\Phi(X_-)$ to $\Phi(X_+)$,
\item 
to the trivial morphism $[[0,1]\times X]$ the identity morphism
$1_{\Phi(X)}$ of $\Phi(X)$;
\end{enumerate} 
and satisfy the following {\em Cerf relations} for any pair of
elementary morphisms $[Y_{1}]$ from $X_0$ to $X_1$ and $[Y_{2}]$ from
$X_1$ to $X_2$ :
\begin{enumerate}
\item If $[Y_{1}] \circ [Y_{2}]$ is a cylindrical morphism, then
$$\Phi([Y_{1}]) \circ \Phi([Y_{2}]) = \Phi([Y_{1}] \circ [Y_{2}]) .$$
\item 
If $[Y_{1}],[Y_{2}]$ are related by critical point switch to two other
elementary morphisms $[Y_{1}'],[Y_{2}']$ from $X_0$ to $X'_1$ and from
$X'_1$ to $X_2$, then
$$\Phi([Y_{1}]) \circ \Phi([Y_{2}]) = \Phi([Y_{1}']) \circ \Phi([Y_{2}']) .$$
\item
If one of $[Y_{1}],[Y_{2}]$ is cylindrical, then
$$\Phi([Y_{1}]) \circ \Phi([Y_{2}]) = \Phi([Y_{1}]\circ[Y_{2}]).
$$
\end{enumerate} 
Then there exists a unique extension of $\Phi$ to a $n+1$-dimensional
connected $\CC$-valued field theory $\ol{\Phi}: \Cob_{n+1}^0 \to \CC$.
\end{theorem} 

\section{Central curvature connections}
\label{svft}

In this section we show that assigning to each closed manifold
resp.\ elementary bordism a moduli space of connections, considered as
a symplectic manifold resp.\ Lagrangian correspondence, gives rise to
a symplectic-valued field theory. The general idea is well-known to
experts, especially in the context of quantum Chern-Simons theory
where according to Witten's suggestion the invariants of
three-dimensional bordisms arise by ``quantizing'' these Lagrangian
correspondences \cite{wi:jo}.

We begin by reviewing the construction of a symplectic category in
Section~\ref{symp}.  Section~\ref{central} provides general background
on moduli spaces of central curvature connections with fixed
determinant.  Then Sections~\ref{centralsurf} and \ref{sec:cobord}
construct a partial symplectic-valued field theory in the sense of
Theorem~\ref{cerfthm}.

\subsection{The symplectic category}
\label{symp} 

The moduli spaces of connections over compact oriented surfaces will
be symplectic manifolds, whose dimension varies with the genus of the
surfaces.  Thus the morphisms between these symplectic manifolds,
associated to elementary bordisms, can not in general be
symplectomorphisms.  However, elementary bordisms define the more
general morphisms between symplectic manifolds introduced by Weinstein
\cite{we:sc} as follows.

\begin{definition} \label{geomcomp} {\rm (Geometric composition of Lagrangian correspondences)} 
For $j = 0,1,2$ let $M_j$ be a symplectic manifold with symplectic
form $\omega_{M_j}$.
\begin{enumerate} 
\item A {\em Lagrangian correspondence} from $M_0$ to $M_1$ is a
  Lagrangian submanifold $L\subset M_0^-\times M_1$ with respect to
  the symplectic structure $(-\omega_{M_0}) \oplus \omega_{M_1}$.
\item 
The {\em geometric composition} of Lagrangian correspondences
$$L_{1}\subset M_0^-\times M_1, \quad L_{2}\subset M_1^-\times M_2$$ 
is the point set
\begin{equation} \label{circ}
L_{1}\circ L_{2} := \pi_{M_0\times M_2} \bigl( ( L_{1} \times L_{2} )
\cap ( M_0 \times \Delta_{M_1} \times M_2 ) \bigr) \subset M_0\times
M_2 .
\end{equation} 
\item 
A geometric composition is called {\em transverse} if the intersection
in \eqref{circ} is transverse (and hence smooth) and {\em embedded} if
the projection $\pi_{M_0\times M_2}$ restricts to an injection of the
smooth intersection.  In that case the image of the intersection is a
smooth Lagrangian correspondence $L_{1} \circ L_{2} \subset M_0^-
\times M_2$.
\end{enumerate} 
\end{definition}

In general, even after a generic perturbation, the fiber product
$L_1 \times_{M_1} L_2$ is at most immersed by projection onto
$M_0 \times M_2$. In order to be able to work more easily with
holomorphic curves, however, we wish to have embedded Lagrangians.  In
\cite{we:co} and \cite{Ainfty} we thus constructed a symplectic
category for Lagrangian correspondences using a more algebraic
composition, as follows.

\begin{definition}  {\rm (Algebraic composition of Lagrangian correspondences)} 
\begin{enumerate}
\item A {\em generalized Lagrangian correspondence} $\ul{L}$ from
  $M_-$ to $M_+$ (both symplectic manifolds) of length $m \ge 0$
  consists of
\begin{enumerate}
\item a sequence $N_0,\ldots,N_m$  of
  symplectic manifolds with $N_0 = M_-$ and $N_m = M_+$; and 
\item 
a sequence $\ul{L} = (L_{1},\ldots, L_{m})$ of compact 
  Lagrangian correspondences with $L_{j} \subset N_{j-1}^-\times 
  N_{j}$ for $j=1,\ldots,m$.
\end{enumerate}
Here we allow empty symplectic manifolds or Lagrangian correspondences.
\item 
The {\em algebraic composition} of generalized Lagrangian
correspondences $\ul{L}$ and $\ul{L}'$ is given by concatenation
$\ul{L}\concat\ul{L}'=(L_{1},\ldots, L_{m},L_{1}',\ldots, L_{m'}')$.
\end{enumerate} 
\end{definition} 

\begin{definition} {\rm (Symplectic category)}  
Let $\Symp^{\sharp}$ be the category whose
\begin{enumerate}
\item objects are symplectic manifolds $(M,\omega)$ (in particular, we
  include the empty manifold $\emptyset$);
\item morphisms $\Hom(M_-,M_+)$ of $\Symp^{\sharp}$ are compact
  oriented generalized Lagrangian correspondences $\ul{L}$ from $M_-$
  to $M_+$ equipped with gradings and relative spin structure modulo
  the {\em composition equivalence} relation $\sim$ generated by
\begin{equation} \label{compequiv}  
\bigl(\ldots,L_{j},L_{j+1},\ldots \bigr) \sim \bigl(\ldots,L_{j}\circ
L_{j+1},\ldots\bigr)
\end{equation} 
for all sequences and $j$ such that $L_{j} \circ L_{j+1}$ is
transverse and embedded; we also set $\Delta_M \sim \emptyset $ where
$\Delta_M \subset M^- \times M$ is the diagonal and $\emptyset$ is the
empty sequence.
\item  composition of morphisms $[\ul{L}]\in\Hom(M,M')$ and
$[\ul{L}']\in\Hom(M',M'')$ is defined by
$$
[\ul{L}]\circ[\ul{L}'] := [\ul{L}\concat\ul{L}'] \;\in\Hom(M,M'') .
$$
\end{enumerate}
An identity $1_M\in\Hom(M,M)$ is given by the equivalence class of the
empty sequence $[\emptyset]$ of length zero (since concatenating with
the empty set is the identity on sequences) or equivalently, the
equivalence class $1_M:=[\Delta_M]$ of the diagonal.
\end{definition}

\begin{remark} \label{graphs}
Let $\Symp$ be the category of (symplectic manifolds,
symplectomorphisms).  There is a canonical functor
$$ \Symp \to \Symp^{\sharp}, \quad \left( \begin{array}{cc} M \mapsto M
  \\ \phi \mapsto [(\on{graph}(\phi))] \end{array} \right) .$$
That is, the identity maps to the diagonal correspondence and the
composition of graphs is the graph of the composition of two
symplectomorphisms.  Indeed, the geometric composition 
\begin{equation} \label{graphcomp}
\graph(\varphi_2) \circ \graph(\varphi_1) =  \graph(\varphi_1 \circ
\varphi_2) 
\end{equation} 
(reversal of order due to conventions) of graphs of symplectomorphisms
$\varphi_1,\varphi_2$ is always transverse and embedded.
\end{remark} 

\begin{definition} \label{sympequiv} {\rm (Symplectomorphism equivalences)}  
Let $\ul{L},\ul{L}'$ be generalized Lagrangian correspondences from
$M_-$ to $M_+$ of the same length $m$.  A {\em symplectomorphism
  equivalence} from $\ul{L}$ to $\ul{L}'$ is a collection of
symplectomorphisms $(\varphi_j: N_j \to N_j')_{j=0,\ldots,m}$ such that
$$\varphi_0 = \Id_{M_-}, \quad \varphi_m = \Id_{M_+}, \quad L_{j}' =
(\varphi_{j-1} \times \varphi_j)(L_{j}), \ \ j = 1,\ldots,m .$$
\end{definition}

\begin{lemma} \label{induces} {\rm (Composition equivalences via symplectomorphism
equivalences)} 
Any symplectomorphism equivalence from $\ul{L}$ to $\ul{L}'$ induces
 a composition equivalence $\ul{L}\sim\ul{L}'$.
\end{lemma} 

\begin{proof}  Consider the commutative diagram 
$$ \begin{diagram} \node{M_- = N_0} \arrow{e,t}{L_{1}}
    \arrow{s,l}{\varphi_0=\Id} \arrow{se} \node{N_1} \arrow{e,t}{L_{2}}
    \arrow{s,r}{\varphi_1} \arrow{se} \node{\ldots} \arrow{e,t}{L_{m}}
    \arrow{s,r}{\varphi_{m-1}} \arrow{se} \node{N_m = M_+}
    \arrow{s,r}{\varphi_m=\Id} \\ \node{M_- = N_0'} \arrow{e,t}{L_{1}'}
    \node{N_1'} \arrow{e,t}{L_{2}'} \node{\ldots} \arrow{e,t}{L_{m}'}
    \node{N_m' = M_+} \end{diagram}.
 $$
Each diagonal morphism can be written in two ways as the composition
of a Lagrangian correspondence with the graph of a symplectomorphism,
$$ L_{j} \circ \mgraph(\varphi_{j}) = \mgraph(\varphi_{j-1}) \circ
L'_{j} .
$$
A composition equivalence from $\ul{L}$ to $\ul{L}'$ is defined by
first replacing $L_{m}$ with $(\mgraph(\varphi_{m-1}), L_{m}')$, then
iteratively replacing $(L_{j}, \mgraph(\varphi_{j}))$ with
$(\mgraph(\varphi_{j-1}),L'_{j})$ for $j=m-1$ to $j=2$, and eventually
replacing $(L_{1}, \mgraph(\varphi_{1}))$ with $L'_1$.
\end{proof}

For the purposes of Floer theory, we will need our symplectic
manifolds and correspondences to carry additional structures and
satisfy additional hypotheses.  

\begin{definition}  \label{admiss} {\rm (Monotone symplectic manifolds)}  
For any monotonicity constant $\tau>0$ we introduce the following
admissible classes of symplectic manifolds and generalized Lagrangian
correspondences.

\begin{enumerate}
\item
 A symplectic manifold $(M,\omega)$ is {\em monotone} with
  monotonicity constant $\tau $ if $\tau c_1(M) = [\omega]$ in
  $H^2(M)$.

\item A symplectic manifold $(M,\omega)$ is {\em $\tau$-admissible} if
  it is compact, monotone with monotonicity constant $\tau$, and has
  even minimal Chern number:
$$ \lan c_1(M) , H_2(M,\Z) \ran \subset 2 \Z .$$
\item A generalized Lagrangian correspondence
  $\ul{L} = (L_1,\ldots, L_m)$ from $M_-$ to $M_+$ is {\em admissible}
  if each Lagrangian correspondence in the sequence $L_i$ is
  simply-connected, compact, oriented, and relatively spin.
\end{enumerate}
Let $\Symp^{\sharp}_{\tau,N}$ denote the category whose
\begin{enumerate} 
\item[(i)] objects are $\tau$-admissible symplectic manifolds $M$
  equipped with $N$-fold Maslov covers $\Lag^N(M) \to \Lag(M)$ in the
  sense of Seidel \cite{se:gr};
\item[(ii)] morphisms from $M_0$ to $M_1$ are equivalence classes
  $[\ul{L}]$ of admissible generalized Lagrangian correspondences from
  $M_0$ to $M_1$, each correspondence in the sequence equipped with
  relative spin structures and $N$-gradings, and where the composition
  equivalence relation $\sim$ is that generated by \eqref{compequiv}
  restricted to such correspondences; and
\item[(iii)] composition of morphisms $\ul{L}_{01}, \ul{L}_{12}$ is
  defined as concatenation $\ul{L}_{01} \# \ul{L}_{12}$ as before.
\end{enumerate}
\end{definition} 

\begin{remark} 
\begin{enumerate}
\item {\rm (Other possible assumptions on correspondences)} The
  condition in (c) can be replaced with other conditions that
  guarantee monotonicity, or just requiring monotonicity itself, but
  in practice in this paper we just check simply-connectedness.
\item {\rm (Inclusion of monotone symplectic categories in the
  symplectic category)} There is a canonical functor
  $\Symp^{\sharp}_{\tau,N} \to \Symp^{\sharp}$ induced by inclusion of
  objects 
  $$ \on{Ob}(\Symp^{\sharp}_{\tau,N}) \hookrightarrow \on{Ob}(
  \Symp^{\sharp})$$
  and on morphisms
  \begin{equation} \label{mormap} \on{Hom}(\Symp^{\sharp}_{\tau,N})
    \to \on{Hom}( \Symp^{\sharp}), \quad [\ul{L}]_{\tau,N} \to
    [\ul{L}] \end{equation}
mapping equivalence classes $[\ul{L}]_\tau$ of generalized admissible
Lagrangian correspondences to equivalence classes $[\ul{L}]$ of
generalized Lagrangian correspondences.  However, the map on morphisms
\eqref{mormap} may not be an inclusion since two admissible
generalized Lagrangian correspondences may be equivalent through a
non-admissible generalized Lagrangian correspondence.
\end{enumerate} 
\end{remark} 

\subsection{Moduli spaces of central curvature connections with fixed determinant}  
\label{central}

This section introduces the moduli spaces of central curvature
connections with fixed determinant on compact manifolds of any
dimension, for a general class of structure groups.

\subsubsection{Central curvature connections and gauge transformations}

We first introduce some notation for connections and gauge
transformations.  Let $G$ be a compact, connected Lie group with Lie
algebra $\g$.  Let $Z$ denote the center of $G$ and $\z$ its tangent
space in $\g$.  Let $X$ be a compact oriented manifold of dimension
$\dim(X) = n$ with (possibly empty) boundary.  Let $\pi: \, P \to X$
be a principal $G$-bundle.

\begin{definition} 
\begin{enumerate} 
\item {\rm (Associated vector bundles)} For any finite-dimensional
  real $G$-representation $V$ we denote by $P(V) = (P \times V)/G$ the
  {\em associated vector bundle}, where $G$ acts on $P \times V$ by
  $g(p,v) = (pg^{-1},gv)$.  Denote by
$$\Omega(X,P(V)) := \bigoplus_{k = 0}^n \Omega^k(X,P(V)) $$
the space of forms with values in $P(V)$.
 \item {\rm (Adjoint bundles)} In particular, $P(\g) = (P \times
   \g)/G$ is the {\em adjoint bundle} associated to the adjoint
   representation of $G$ on $\g$.
\item {\rm (Splittings of the adjoint bundles)} Any invariant subspace
  $\h \subset \g$ induces an inclusion $P(\h) \subset P(\g)$.  The
  splitting $\g = [\g,\g] \oplus \z$ into the semisimple and central
  parts of $\g$ induces a splitting of the adjoint bundle,
\begin{equation} 
\label{split2}
P(\g) = P([\g,\g]) \oplus P(\z).
\end{equation} 
In the case $G=U(r)$ to which we will specialize later, the center is
given by the diagonal matrices $Z=U(1){\rm Id}$, and the splitting is
$$\u(r)=\su(r)\oplus\u(1){\rm Id} .$$
\item {\rm (Affine space of connections)} Let $\A(P)$ be the space of
  connections on $P$,
$$ \A(P) = \left\{ \alpha \in \Omega^1(P,\g) \left| \begin{array}{ll}
  \alpha(\xi_P) = \xi & \forall \xi \in \g, \\ \alpha(vg) =
  \Ad(g^{-1}) \alpha(v) & \forall v \in TP, g \in G \end{array}
\right. \right\} .$$
Here $\xi_P \in \Vect(P)$ denotes the vector field generated by the
action of $\xi \in \g$. 
\item {\rm (Basic forms)} For any non-negative integer $k$ the space
  $\Omega^k(X,P(\g))$ of $k$-forms with values in $P(\g)$ is
  isomorphic via $\pi^*$ to the space $\Omega^k(P,\g)_{\on{basic}}$ of
  basic (that is, equivariant and horizontal) $k$-forms.  With this
  notation $\A(P)$ is an affine space modelled on $\Omega^1(X,P(\g))$.
  That is, $\Omega^1(X,P(\g))$ acts on $A(P)$ faithfully transitively
  by $ \alpha \mapsto \alpha + \pi^* b$ for $\alpha \in \A(P)$, $b \in
  \Omega^1(X,P(\g)) $.
\item {\rm (Curvature)} The curvature of $\alpha \in \A(P)$ is the two
  form $F_\alpha \in \Omega^2(X,P(\g))$ defined by
$$ \pi^* F_\alpha = \d \alpha +
\tfrac 12 [\alpha\wedge\alpha] \in \Omega^2(P,\g)_{\on{basic}} .$$  
\item {\rm (Covariant derivative)} The covariant derivative in the
adjoint representation is
$$\d_\alpha : \Omega^{*}(X,P(\g)) \to \Omega^{{*} +1
}(X,P(\g)), \ \ \ \pi^* \d_\alpha \beta = \d \pi^* \beta + [\alpha\wedge \pi^* \beta] .$$
\item {\rm (Bianchi identity)} $\d_\alpha F_\alpha = 0$.
\item {\rm (Central curvature connections)} A connection $\alpha$ is
  {\em central curvature} if $F_\alpha$ takes values in $P(\z) \subset
  P(\g)$, that is, $F_\alpha^{[\g,\g]} = 0$.
\item {\rm (Group of gauge transformations)} Let
$$ \G(P) = \bigl\{ \phi: P \to P \, \big| \ \pi \circ \phi = \pi,
\ \phi(pg) = \phi(p)g \ \ \forall \ p \in P, \ g \in G \big\} $$
denote the group of gauge transformations, that is, $G$-equivariant
automorphisms of $P$.  
\item {\rm (Action of gauge transformations on connections)} The group
  $\G(P)$ acts on the left on $\A(P)$ by
$$ \G(P) \times \A(P) \to \A(P), \quad (\phi,\alpha) \mapsto
 (\phi^{-1})^* \alpha .$$
\item {\rm (Infinitesimal gauge transformations)} The Lie algebra of
  $\G(P)$ can be identified with $\Omega^0(X,P(\g))$ by associating
  the vector field $p \mapsto\frac \d{\d t}|_{t=0} p\exp(t\xi)$ to
  $\xi\in\Omega^1(P,\g)_{\rm basic}=\Omega^0(X,P(\g))$.  With this
  notation, the infinitesimal action of $\G(P)$ on $\A(P)$ is given by
\begin{equation} \label{inf action}
 \Omega^0(X,P(\g)) \times \A(P) \to \Omega^1(X,P(\g)), \quad (\xi, \alpha)
\mapsto  \d_\alpha \xi .
\end{equation}
\item {\rm (Action of gauge transformations on covariant derivatives)}
  The action of $\phi\in\G(P)$ on $P$ induces an action on $P(V)$ for
  any $G$-representation $V$, denoted by $\phi_V$.  The covariant
  derivative $\d_\alpha$ and curvature $F_\alpha$ transform as 
$$ \d_{\phi \alpha} = \phi_\g
  \d_\alpha \phi_\g^{-1} , \quad  F_{\phi
    \alpha} = \phi_{\g} F_\alpha .$$ 
Hence the subset of central curvature connections is invariant under
gauge transformations.
\end{enumerate}
\end{definition} 

\begin{remark} 
\begin{enumerate} 
\item {\rm (Splitting of the covariant derivative)} Using the
  splitting \eqref{split2} we write
$$ \d_\alpha = \d_\alpha^\z \oplus \d_\alpha^{[\g,\g]} \qquad
\text{and}\qquad F_\alpha = F_\alpha^\z \oplus F_\alpha^{[\g,\g]} .$$
\item {\rm (Basic inner product)} Let $ \langle \cdot , \cdot \rangle$
  denote the basic invariant inner product on $\g$, see
  \cite[p.49]{ps:lg}.  For $G$ simple, this is the unique inner
  product such that the norm-square of the highest coroot is~$2$.  For
  example, in the case that $G = U(r)$, the basic inner product is
  $\lan \xi, \zeta \ran = - \Tr(\xi \zeta)$ for $\xi,\zeta \in \u(r)$.
\label{basic}
\item {\rm (Adjoint of the covariant derivative)} A choice of metric
  on $X$ induces a Hodge star operator
$$*:\Omega^{k}(X,P(\g)) \mapsto \Omega^{n - k}(X,P(\g)), \quad k =
  0,\ldots, n .$$
Together with the inner product on $\g$ this induces a metric on each
$\Omega^k(X,P(\g))$.  The formal adjoint of the covariant derivative
is
$$ \d_\alpha^* : \Omega^\bullet(X,P(\g)) \to \Omega^{\bullet
  -1}(X,P(\g)), \ \ \beta \mapsto - (-1)^{(n- \bullet)(\bullet - 1)} *
\d_\alpha * \beta.$$
\end{enumerate} 
\end{remark}

\subsubsection{Linear theory}

In this subsection we discuss the cohomology of the covariant
derivative associated to a central curvature connection.  We review
the well-known fact that in certain dimensions the cohomology of a
compact oriented manifold with boundary $X$ restricts to a Lagrangian
subspace in the middle dimensional cohomology on the boundary
$\partial X$.

\begin{definition}  Let $P \to X$ be a principal $G$-bundle.  
\begin{enumerate} 
\item {\rm (Cohomology of a central curvature connection)} Let $\alpha
  \in \A(P)$ be a central curvature connection.  Since its curvature
  $F_\alpha$ is central we have $\d_\alpha^2 = 0$.  Define cohomology
  groups
$$ H(X;\d_\alpha) = \bigoplus_{k = 0}^n H^k(X;\d_\alpha),
  \ \ \ H^k(X;\d_\alpha) = \frac{\ker(\d_\alpha |
    \Omega^k(X,P(\g)))}{\im(\d_\alpha | \Omega^{k-1}(X,P(\g)))} .$$
\item {\rm (Relative cohomology of a central curvature connection)}
  Denote restriction to the boundary by 
  \begin{equation} \label{defrho} \rho: \Omega(X,P(\g)) \to
    \Omega(\partial X; P(\g) |_{\partial X}) .\end{equation}
  The map $\rho$ is a cochain map and so induces a map of cohomologies
  $ H(X; \d_\alpha) \to H(\partial X ; \d_{\alpha | \partial X}) .$
  Let $H(X,\partial X;\d_\alpha)$ be the relative cohomology groups
  consisting of forms whose tangential part to the boundary vanishes,
$$ H(X,\partial X;\d_\alpha)
 = \ker(\d_\alpha | \ker \rho)/
\im(\d_\alpha | \ker \rho) .$$
\item {\rm (Hodge isomorphisms)} The Hodge isomorphisms on manifolds
  with boundary (see e.g.\ \cite[Section 4.1]{gi:inv}) give
\begin{align}\label{hodge isom}
 H(X;\d_\alpha) &\cong \ker(\d_\alpha) \cap \im( \d_\alpha
)^\perp = \ker ( \d_\alpha \oplus \d_\alpha^* \oplus \rho * ) , \\
H(X,\partial X;\d_\alpha) &\cong \ker(\d_\alpha\oplus\rho) \cap \im( \d_\alpha|_{\ker\rho} )^\perp = 
\ker ( \d_\alpha \oplus \d_\alpha^* \oplus \rho  ).  \nonumber
\end{align}
Using the first identification we have a Hodge decomposition 
\begin{equation} \label{hodge}
 \Omega(X,P(\g)) = \im(\d_\alpha) \oplus H(X;\d_\alpha)
\oplus \im(\d_\alpha^* | \ker \rho * ) .
\end{equation}
Here the direct sum holds in any $L_2^s$-Sobolev completion for $s\geq
0$ (that is we use the norm $\bar{H}^s$ of \cite{ho:an3} measuring $s$
fractional derivatives in $L^2$), as a consequence of the following
elliptic estimate from e.g.\ \cite[Section 20.1]{ho:an3}.  Let
$\Pi_{H^k}$ denote the $L^2$-orthogonal projection onto
$H^k(X;\d_\alpha)\subset\Omega^k(X,P(\g))$.  There is a constant $C$
depending on $\alpha$ such that for all $\eta\in \Omega^k(X,P(\g))$
\begin{multline} \label{hodge estimate}
\| \eta \|_{L_2^{s+1}(X)} \leq C \bigl( \| (\d_\alpha + \d_\alpha^*)\eta \|_{L_2^{s}(X)} + \|\rho*\eta
\|_{L_2^{s+ 1/2}(\partial X)} + \|\Pi_{H^k}\eta \|_{L_2(X)} \bigr) .
\end{multline}
\item {\rm (Non-degeneracy of the duality pairings)} Using these Hodge
  isomorphisms, the Hodge star induces a duality isomorphism
\begin{equation} \label{Hodge duality} 
H^{*}(X;\d_\alpha) \to H^{n - {*}}(X,\partial X;\d_\alpha), \ \ \
 \beta \mapsto * \beta .
\end{equation}
This duality implies non-degeneracy of the pairing
\begin{equation} \label{Hodge pairing}
 H^{*}(X;\d_\alpha) \times H^{n - {*}}(X,\partial X;\d_\alpha) \to
 \R, \ \ \ (\gamma,\beta) \mapsto \int_X \langle 
\gamma \wedge \beta  \rangle .\end{equation}
Replacing $X$ with $\partial X$ we have a non-degenerate pairing on $
H(\partial X;\d_{\alpha|_{\partial X}}) $.  
\end{enumerate} 
\end{definition} 

If $\dim(X) = 2k+1$ is odd, the pairing in the last item restricts to
a non-degenerate pairing on the middle dimensional homology
$H^k(\partial X;\d_{\alpha |_{\partial X}})$.  The pairing is
skew-symmetric, hence symplectic, if $k$ is odd. In that case, the
following Lemma shows that the restriction of $H(X;\d_\alpha)$ to the
boundary gives rise to a Lagrangian subspace.

\begin{lemma} \label{maxiso} {\rm (Lagrangian restriction of odd cohomology)}  
Let $\alpha \in \A(P)$ be a central curvature connection.  The image
$\rho ( H(X;\d_\alpha))$ of $H(X;\d_\alpha)$ in $ H(\partial
X;\d_{\alpha |_{\partial X}})$ is maximally isotropic.  In particular,
if $\dim(X) = 2k+1$ for some integer $k$ then $\rho (
H^k(X;\d_\alpha))\subset H^k(\partial X;\d_{\alpha |_{\partial X}})$
is maximally isotropic, and if $k$ is odd,
\begin{equation} \label{half} \dim \bigl( \rho ( H^k(X;\d_\alpha) ) \bigr) = \tfrac 12 \dim \bigl( H^k(\partial X;\d_{\alpha |_{\partial X}}) \bigr) .
\end{equation}
\end{lemma}

\begin{proof} Stokes' theorem implies that the image
$\rho ( H(X;\d_\alpha))$ is isotropic: For $\beta_1,\beta_2 \in
  \ker(\d_\alpha)$ 
\begin{eqnarray*}
 \int_{\partial X} \lan \rho(\beta_1) \wedge \rho(\beta_2) \ran
 &=& \int_X
  \d \lan \beta_1 \wedge \beta_2 \ran \\
 &=& \int_X
  \lan \d_\alpha \beta_1 \wedge \beta_2 \ran 
+ (-1)^{|\beta_1|}   \lan \beta_1 \wedge \d_\alpha \beta_2 \ran  
= 0 .\end{eqnarray*}
To see that the image is maximal isotropic, first note that restriction to
the boundary induces a long exact sequence
\begin{equation} \label{les}
\ldots \to H^{\bullet}(X,\partial X;\d_\alpha) \to
H^{\bullet}(X;\d_\alpha) \to H^{\bullet}(\partial X;\d_{\alpha |_{\partial X}})
\to \ldots .\end{equation}
Now consider the commutative diagram formed from the long exact
sequence \eqref{les} and its dual:
\[ 
\ldots \begin{diagram}
\node{ H^{\bullet}(X;\d_\alpha)} \arrow{e} \arrow{s} 
\node{H^{\bullet}(\partial X;\d_{\alpha |_{\partial X}})} \arrow{s} \arrow{e,t}{c}
\node{ H^{{\bullet}+1}(X,\partial X;\d_\alpha)} \arrow{s}
 \\
\node{ H^{n- {\bullet}} (X,\partial X;\d_\alpha)^\dual}  \arrow{e} 
\node{ H^{n -1 - {\bullet} }(\partial X;\d_{\alpha |_{\partial X}})^\dual} 
\arrow{e} \node{ H^{n- 1 - {\bullet}}(X;\d_\alpha)^\dual} 
\end{diagram} \ldots .
\]
Here the vertical arrows are given by \eqref{Hodge duality} and the
pairing \eqref{Hodge pairing}. To check commutativity, use Stokes'
theorem and the fact that the connecting morphism $c$ above is given
by $\beta\mapsto \d_{\alpha}\tilde\beta$ for any extension
$\tilde\beta|_{\partial X}=\beta$.

Now suppose $\beta \in H^{\bullet}(\partial X;\d_{\alpha |_{\partial
    X}})$ lies in the annihilator of the image of
$H^{n-1-{\bullet}}(X;\d_\alpha) $.  By definition  $\beta \in
H^{\bullet}(\partial X;d_{\alpha |_{\partial X}})$ maps to $0$ in
$H^{n- 1 - {\bullet}}(X;\d_\alpha)^\dual$.  Since the vertical maps
are isomorphisms and the diagram is commutative, the image of $\beta$
in $H^{{\bullet} + 1}(X,\partial X;\d_\alpha)$ vanishes, so by
exactness of the top sequence $\beta$ lies in the image of
$H^{\bullet}(X;\d_\alpha)$.  The last claim follows from the fact that
maximally isotropic subspaces of symplectic vector spaces are
half-dimensional.  For analogous results on Dirac operators on
manifolds with boundary see \cite[Theorem 22.24]{bo:el}.
\end{proof} 

\begin{corollary} \label{c maxiso}
If $X$ has dimension $3$ then $\rho ( H^1(X;\d_\alpha) ) \subset
H^1(\partial X;\d_{\alpha |_{\partial X}})$ is a Lagrangian subspace.
Furthermore, if $H^1(X,\partial X;\d_\alpha) = 0$ then $\rho:
H^1(X;\d_\alpha) \to H^1(\partial X;\d_{\alpha |_{\partial X}})$ is a
Lagrangian embedding.
\end{corollary}

\begin{proof}   The first statement follows immediately from 
Lemma \ref{maxiso} with $k=1$.  The injectivity of $\rho$ in the
second statement follows from the long exact sequence \eqref{les}.
\end{proof} 

\subsubsection{Moduli spaces}
\label{func}

This subsection introduces the moduli space of connections with
central curvature and fixed determinant.  As before, $G$ is a compact
connected group and $X$ a compact oriented manifold with (possibly
empty) boundary.

\begin{definition} \ \ \label{comgroup}
\begin{enumerate} 
\item {\rm (Commutator subgroup)} Since $G$ is compact and connected,
  by Got\^o's theorem, the commutator mapping
  $(a_1,b_1,\ldots,a_g,b_g) \to \prod_{j=1}^g [a_j,b_j]$ is surjective
  onto the commutator subgroup
$$G_0:=[G,G] = \{ [g_1,g_2] \,|\, g_1,g_2 \in G\}.$$
In particular, $[U(r),U(r)] = SU(r)$.  More generally, $G_0$ is the
subgroup whose Lie algebra equals to the semisimple part $[\g,\g]$ of
the Lie algebra $\g$.
\item {\rm (Determinant homomorphism)} The group homomorphism $G \to
  G/G_0$ induces for any principal $G$-bundle $P$ a principal
  $G/G_0$-bundle 
$$\det(P) := P / G_0 \cong P \times_G (G/G_0)
  \overset{\pi'}{\longrightarrow} X .$$  
In the special case $G = U(r)$ the bundle $\det(P)$ is the principal
$U(1)$-bundle induced by the determinant $\det: U(r) \to U(1)$.
\item {\rm (Determinant of a connection)} The homomorphism of Lie
  algebras 
$$\pi_\z: \g \to \g/[\g,\g]\cong\z$$ 
induces a map $ \det : \A(P) \to \A(\det(P)) $.  Indeed for any
connection $\alpha\in\A(P)$ the form $\pi_\z\circ\alpha \in
\Omega^1(P,\z)$ is a basic form for the $G_0$-bundle $P\to\det(P)$ and
descends to a connection form on $\det(P)$.  Given a connection $
\delta \in \A(\det(P))$ denote by
$$
\A_\delta(P) = \bigl\{ \alpha \in \A(P) \, \big| \,  F_\alpha^{[\g,\g]} = 0 , \det(\alpha)=\delta \bigr\}
$$
its inverse image in the space of central curvature connections.
\item {\rm (Gauge transformations fixing the determinant)} The group
  of gauge transformations with trivial determinant $\G_0(P)$ is
  defined to be the kernel of the homomorphism $\G(P) \to
  \G(\det(P))$.  It acts on $\A_\delta(P)$ for any $\delta \in
  \A(\det(P))$.
\item {\rm (Moduli spaces with fixed determinant)} Denote by
$$ {M}_\delta(P) = \A_\delta(P)/ \G_0(P) $$
its quotient, the moduli space of connections with central curvature
and fixed determinant. 
\end{enumerate}
\end{definition} 

\begin{remark} \label{delta dont matter} {\rm (Independence of the moduli spaces
from the choice of determinant)} The spaces $M_\delta(P)$ as $\delta$
  ranges over connections on $\det(P)$ are identified as follows.  The
  vector space $\Omega^1(X,\z)$ acts on $\A(P)$ by $ \alpha \mapsto
  \alpha + \pi^*a$ for $a\in\Omega^1(X,\z)$.  The curvature and
  determinant transform as
$$ F_{\alpha + \pi^* a} = F_\alpha + \d a, \quad \det(\alpha + \pi^*a)
  = \det(\alpha) + \pi'^* a .$$  
Indeed, the defining equation is
$$
\pi_{G_0}^* \det(\alpha + \pi^*a) 
= \pi_\z \circ (\alpha + \pi^*a) 
= \pi_\z \circ \alpha + \pi_{G_0}^* \pi'^*a 
= \pi_{G_0}^* \bigl( \det(\alpha) + \pi'^*a \bigr) .
$$
Each $a\in\Omega^1(X,\z)$ induces an identification of moduli spaces
$$M_\delta(P) \to M_{\delta + \pi'^*a}(P) .$$
This shift provides unique identifications of
$M_{\delta_1}(P), M_{\delta_2}(P)$ for all
$\delta_1,\delta_2 \in \A(\det(P))$.  Indeed, $\A(\det(P))$
is \label{change8} an affine space over $\Omega^1(X,\z)$ via
$\delta\mapsto \delta + \pi'^*a$ for $a\in\Omega^1(X,\z)$.  We will
hence from now on refer to
$$ M(P) := M_\delta(P) = \A_\delta(P)/ \G_0(P) $$
as {\em the} moduli spaces of central curvature connections with fixed
determinant with only minor abuse of language.
\end{remark}

\begin{proposition} \label{smoothmod} {\rm (Condition for smoothness of the moduli space)}  
Let $\alpha \in \A(P)$ be a central curvature connection.  If
$H^0(X;\d_\alpha^{[\g,\g]}) = H^2(X;\d_\alpha^{[\g,\g]}) = 0$ then
$M(P)$ is a finite-dimensional orbifold at $[\alpha]$ with tangent
space isomorphic to $H^1(X;\d_\alpha^{[\g,\g]})$.  If in addition $G =
U(r)$, then $M(P)$ is a finite-dimensional manifold in a neighborhood
of $[\alpha]$.
\end{proposition}

\begin{proof}  
Fix $\delta\in\A(\det(P))$.  Any $\alpha \in \A_\delta(P)$ splits into
semi-simple and central part $\alpha= \alpha_ 1 + \pi_{G_0}^*\delta$,
where $\alpha_{1}\in\Omega^1(P,[\g,\g])$ satisfies 
\begin{equation} \label{solve} \pi^*
F_\alpha^{[\g,\g]}= \d\alpha_{1}+\frac 12 [\alpha_{1},\alpha_{1}] =
0 .\end{equation} 
By standard arguments (as for flat connections), any solution of
\eqref{solve} of Sobolev class $L_2^s$ with $2s>n$ is gauge equivalent to
a smooth solution.  Hence we can think of $M(P)$ as the quotient of
solutions of \eqref{solve} of class $L_2^s$ by the $L_2^{s+1}$-closure
of $\G_0(P)$.  In the first step, we will show that the equation
$\d\alpha_{1}+\frac 12 [\alpha_{1},\alpha_{1}] =0$ for $\alpha_{1}\in
\Omega^1(P,[\g,\g])_s$ cuts out a smooth Banach submanifold whose
tangent space at $\alpha_1$ is the kernel of
$$\d_\alpha^{[\g,\g]}=\d_{\alpha_{1}}:\Omega^1(X,P([\g,\g]))_s \to
\Omega^2(X,P([\g,\g]))_{s-1} .$$  
Here we denote by subscripts such as $\Omega(X,P([\g,\g]))_s$ the
$L_2^s$-completion of spaces of smooth forms such as
$\Omega(X,P([\g,\g]))$, and moreover choose $s\geq 2$.  Then by the
vanishing of $H^2(X;\d_\alpha^{[\g,\g]})$, the Hodge estimate
\eqref{hodge estimate} becomes
$$
\Vert b \Vert_{L_2^{s-1}(X)} \leq C \bigl(
  \Vert \d_{\alpha_{1}} b \Vert_{L_2^{s-2}(X)} 
+ \Vert \d_{\alpha_{1}}^* b \Vert_{L_2^{s-2}(X)} + \Vert * b |_{\partial X} \Vert_{L_2^{s-3/2}(\partial X)} \bigr).
$$
for all $b \in \Omega^2(X,P([\g,\g]))_{s-1}$.
A similar estimate, possibly with a different constant, holds with $\d_{\alpha_1}$ replaced by 
$\d_{\alpha_{1} + \pi^* a}$ for sufficiently small $\Vert a
\Vert_{L_2^s(X)}$.
Hence, applying this new estimate to $b = F_{\alpha_1}+\pi^*a$ and
using the Bianchi identity,
\begin{equation} \label{Falpha}
F_{\alpha_{1} + \pi^* a}^{[\g,\g]}=0 \quad \iff \quad 
 (\d_{\alpha_{1}}^* \oplus \rho *) F_{\alpha_{1} + \pi^* a}^{[\g,\g]}
 = 0  .\end{equation} 
It follows that $\A_\delta(P)$ near $\alpha$ is the set of sums $
\alpha + \pi^* a$ where $a$ is a zero of the map
\begin{equation} \label{themap} \Omega^1(X,P([\g,\g]))_s \to \im( \d_{\alpha_{1}}^* \oplus \rho * ), \ \
 a \mapsto (\d_{\alpha_{1}}^* \oplus \rho *) F_{\alpha_{1} + \pi^*
   a}^{[\g,\g]} .\end{equation}
Here the target 
$$
( \d_{\alpha_{1}}^* \oplus \rho * ) \Omega^2(X,P([\g,\g]))_{s-1} \subset 
\Omega^1(X,P([\g,\g]))_{s-2} \times \Omega^{n-2}(\partial X, 
P([\g,\g]) | \partial X)_{s-\frac 32} $$ 
is closed.  Indeed, note that $\im(\d_{\alpha_{1}})^\perp =
\ker(\d_{\alpha_1}^* \oplus \rho *)$.  This implies
$$\im( \d_{\alpha_{1}}^*
\oplus \rho * ) = \im( (\d_{\alpha_{1}}^* \oplus \rho *
)\d_{\alpha_1}) \cong \im(( \d_{\alpha_1} \oplus \d_{\alpha_1}^*
\oplus \rho*) \d_{\alpha_1})  ,$$ 
where the latter isomorphism holds since $\d_{\alpha_1}^2 = 0$.  Now
$\im(\d_{\alpha_1})$ is a closed subspace of
$\Omega^2(X,P([\g,\g]))_{s-1}$, by \eqref{hodge}, and the Hodge
estimate implies that the image of $\im(\d_{\alpha_1})$ under
$\d_{\alpha_1} \oplus \d_{\alpha_1}^* \oplus \rho*$ is also closed, by
\cite[Lemma A.1.1]{ms:jh}.

With this setup, the linearized operator $(\d_{\alpha_{1}}^* \oplus
\rho *)\d_{\alpha_{1}}$ of \eqref{themap} is surjective and has kernel
that of $\d_{\alpha_1}$ since $\ker(\d_{\alpha_1}^* \oplus \rho
*)=\im(\d_{\alpha_{1}})^\perp $. Hence the implicit function theorem
provides a smooth map from the formal tangent space
$\ker(\d_{\alpha_1})$ to its complement in \eqref{hodge}
$$ \ker(\d_{\alpha_{1}}) \to \im(\d_{\alpha_{1}}^* | \ker \rho *), 
\ \ a \mapsto b(a) $$
such that the map
$$ \ker(\d_{\alpha_{1}}) \to \A_\delta(P), \ \ a \mapsto \alpha +
\pi^*(a + b(a)) $$
is a local chart for $\A_\delta(P)$.  

To construct the orbifold structure on the quotient, we show that
$\alpha + \ker(\d_{\alpha_{1}}^* \oplus \rho *)$ is a local slice for
the action of $\G_0(P)$.  The assumption $H^0(\d_\alpha^{[\g,\g]}) =
0$ ensures that the local slice conditions
$$\d_{\alpha_{1}}^*((a+b(a))=0, \quad *((a+b(a))|_{\partial X}=0$$
are transverse to $\ker(\d_{\alpha_{1}})$.  To see that the
stabilizers are finite, note that any automorphism of a bundle with
connection is determined by its restriction to a point, and so the
stabilizer embeds into $G$.  The stabilizer is discrete by vanishing
of $H^0$, and so is a finite subgroup of $G$.  Standard arguments
(e.g.\ \cite[Lemma 4.2.4]{do:fo}) show that $M(P) =
\A_\delta(P)/\G_0(P)$ is Hausdorff.  Hence $M(P)$ is a smooth
orbifold.

If all the stabilizers are the central-valued gauge transformations,
the quotient is \label{change9} in fact a manifold.  Consider the
subgroup $\G_0^{\rm central}(P) \subset\G_0(P)$ given by the central
automorphisms of the form $g(p)=p z(p)$ for some $z:P\to Z$.  If all
stabilizers of $\alpha$ are contained in $\G_0^{\rm central}(P)$, then
$M(P)$ is a smooth manifold near $[\alpha]$.  The subgroup
$\G_0^{\rm central}(P)$ acts trivially on $\A_\delta(P)$. On the other
hand, $\G_0^{\rm central}(P)$ is finite since (by connectedness of $X$
and $G$) the map $z\equiv z^{ss}\in Z^{\rm ss}=Z\cap G_0$ is
constant. Hence we can also realize $M(P)$ as the quotient
$$ M(P) = \A_\delta(P)/ (\G_0(P)/\G_0^{\rm central}(P)) .$$
Since slices exist, this quotient has a natural manifold structure on
the locus where $\G_0(P)/\G_0^{\rm central}(P) $ acts freely.

In the case of the unitary group, the condition of the previous
paragraph is automatically satisfied.  Indeed the vanishing of
$H^0(X;\d_\alpha^{[\g,\g]})$ implies that there are no non-central
automorphism of $\alpha$ in $\G_0(P)$.  For suppose that
$g\in\G_0(P)\setminus \G_0^{\rm central}(P)$ is in the stabilizer of
$\alpha$.  Consider the induced connection $\alpha$ on $P(\C^r)$.  The
gauge transformation $g$ induces an automorphism of $P(\C^r)$ whose
action on some fiber has at least two different eigenvalues; by
parallel transport (independent of the choice of path since the
holonomy commutes with $g$) we obtain a splitting
\begin{equation} \label{PC} P(\C^r)=E_1\oplus E_2, \quad \rank(E_1) =
  r_1 > 0, \quad \rank(E_2) = r_2 > 0 .\end{equation} 
Hence there is a one-parameter family
\begin{equation} \label{opf}
\bigl(e^{it/r_1}{\rm Id}_{E_1} \oplus e^{-it/r_2}{\rm
  Id}_{E_2}\bigr)_{t\in\R} \in \Aut(P(\C^r)), \quad t \in \R \end{equation}
of automorphisms of $\alpha$ on $P(\C^r)$. Since $P$ is the frame
bundle of $P(\C^r)$, this would imply a one-parameter family of
non-central automorphisms of $\alpha$ on $P$, contradicting
$H^0(X;d_\alpha^{[\g,\g]})=0$.
\end{proof}

The moduli spaces $M(P)$ can be described in terms of spaces of
representations of the fundamental group, up to a twist which is
determined by the determinant bundle: 

\begin{definition} {\rm (Adjoint moduli spaces)}  Let $\Ad(G) = G/Z$ and $Z^{\ss}
= Z \cap G_0$.  (In case $G=U(r)$ this means
$\Ad(G)=PSU(r)=SU(r)/e^{2\pi i \Z/r}{\rm Id}$ and $Z^{\ss}=e^{2\pi i
  \Z /r} {\rm Id}$.)
Denote by
$$M_{\rm Ad}(X) := \Hom(\pi_1(X),\Ad(G))/\Ad(G) .$$
the moduli space of representations of $\pi_1(X)$ in $\Ad(G)$, up to
conjugacy.  As explained in e.g. Atiyah-Bott \cite{at:mo}, the space
$M_{\rm Ad}(X)$ is the union of the moduli spaces $M_{\rm Ad}(P)$ of
flat connections on $P$ as $P$ ranges over $\Ad(G)$-bundles.
\end{definition} 

\begin{lemma} \label{pides} {\rm (Relation to moduli of flat bundles)}  
Let $P \to X$ be a principal $G$-bundle.  The moduli space $M(P)$ has
the structure of a topological principal
$\Hom(\pi_1(X),Z^{\ss})$-bundle over a component $M_{\rm Ad}(P)$ of
$M_{\rm Ad}(X)$.  In particular $M(P)$ is compact.
\end{lemma}  

\begin{proof}
Recall that the space of central curvature connections with fixed
determinant is
$$\A_\delta(P) = \left\{ \alpha \in \Omega^1(P,\g)^G \ \left| 
\ \begin{array}{l} \alpha(\xi_P) = \xi \ \forall \xi \in \g
\\ \det(\alpha) = \delta \\ F_\alpha^{[\g,\g]} =0 \end{array} \right. \right\}
,$$
where $\Om^1(\ldots)^G$ denotes the equivariant forms and $\xi_P \in
\Vect(P)$ is the vector field generated by $\xi \in \g$.  The exact
sequence of groups 
$$1 \to Z \to G \to G/Z =: \Ad(G) \to 1 $$ 
induces a splitting of Lie algebras $\g =\z \oplus [\g,\g]$.  The
$\z$-component of any $\alpha\in\A_\delta(P)$ is uniquely determined
by $\det(\alpha) = \delta$.  Hence the projection to the
$[\g,\g]$-component induces a homeomorphism between $\A_\delta(P)$ and
the space
$$ \left\{ \alpha \in \Omega^1(P,[\g,\g])^G \left| \begin{array}{lll} \alpha(\xi_P) = \xi \ \forall
\xi \in [\g,\g] \\ \alpha(\xi_P) = 0 \ \forall \xi \in \z \\ d\alpha +
\tfrac 12 [\alpha\wedge \alpha] =0 \end{array} \right. \right\} $$ 
of flat, equivariant $[\g,\g]$-forms on $P$ that are horizontal with
respect to $\z$.  The latter forms descend to $\Omega^1(P/Z,[\g,\g])$.
Hence we obtain a homeomorphism between $\A_\delta(P)$ and the space
of flat $\Ad(G)$-connections on $P/Z$,
$$\A_{\rm flat}(P/Z) =  \left\{\alpha \in \Omega^1(P/Z,[\g,\g])^{G/Z} \ \left| \begin{array}{l} \alpha(\xi_{P/Z}) =
\xi \ \forall \xi \in [\g,\g] \\  F_\alpha = 0 \end{array} \right. \right\} .$$
While $\G_0(P)$ acts on both these spaces, the larger group of gauge
transformations $\G(P/Z)$ acts on $\A_{\rm flat}(P/Z)$.  Consider
the short exact sequence 
$$1 \to Z^{\ss} \to [G,G]
\to G/Z \to 1 .$$
There is an isomorphism
$$\G(P/Z)/\G_0(P) \cong \Hom(\pi_1(X),Z^{\ss})$$ 
given by viewing the group of gauge transformations as sections of the
bundle $P \times_G (G/Z)$ resp.  $P \times_G [G,G] $.  On the other
hand, any gauge transformation in $\G(P/Z)$ that fixes a flat
connection in $\A_{\rm flat}(P/Z)$ automatically lifts to an element
of $\G_0(P)$.  Hence the projection
$$M(P) \to M_{\rm Ad}(P) := \A_{\rm flat}(P/Z)/\G(P/Z)$$ 
is a $\Hom(\pi_1(X),Z^{\ss})$-principal bundle.  Finally, $M_{\rm
  Ad}(P)$ is homeomorphic to a component of the representation space
$M_{\rm Ad}(X)$, given by those representations which lift to
connections on $P$.  Compactness follows since $\Ad(G), Z^{\ss}$ are
compact and $\pi_1(X)$ is finitely generated.
\end{proof} 

\subsubsection{Moduli spaces for compositions of bordisms}
\label{subsec mod cob}

In this section we study moduli spaces for bordisms of bundles and the
associated gluing law.  Suppose that $(Y,\phi)$ is a compact oriented
connected bordism between compact oriented connected manifolds
$X_\pm$.  Let $(Q,\psi)$ be a bundle bordism between bundles
$P_\pm \to X_\pm$. That is, $Q$ is a bordism from $P_-$ to $P_+$ with
the structure of a $G$-bundle over $Y$ equipped with an identification
$\psi: \partial Q \to \ol{P}_- \cup P_+$ that is an isomorphism of
$G$-bundles.

\begin{definition} \label{corresp} 
\begin{enumerate} 
\item {\rm (Restriction to the boundary)} Denote the pullback map on
  connections by $\rho: \A(Q) \to \A(P_-) \times \A(P_+)$.  The map
  $\rho $ preserves the central curvature and fixed determinant
  conditions (with respect to appropriate restrictions) and is gauge
  equivariant.  Hence $\rho$ induces a map, also denoted
$$ \rho: {M}(Q) \to {M}(P_-) \times M(P_+) .$$
\item 
{\rm (Correspondences associated to bordisms)} Denote by $L(Q)$ the
image of $M(Q)$ in $M(P_-) \times M(P_+)$,
$$ L(Q) = \rho( M(Q)) \subset M(P_-) \times M(P_+)  .$$  
Thus $L(Q)$ is a topological correspondence from $M(P_-)$ to $M(P_+)$,
that is, a subspace of the product.  In our applications, $L(Q)$ will
be a Lagrangian correspondence between symplectic manifolds $M(P_-),
M(P_+)$.    
\end{enumerate} 
\end{definition} 

\begin{remark} {\rm (Compatibility with maps to representation varieties)}  
 Let $x_\pm \in X_\pm , y \in Y$ be base points.  The inclusion of the
 boundary $X_\pm \to Y$ induces a map $\pi_1(X_\pm,x_\pm) \to
 \pi_1(Y,y)$ depending on a choice of path from $x_\pm$ to $y$ up to
 conjugacy.  The map $\rho$ is compatible with the bundle structure
 described in Lemma \ref{pides} in the sense that the diagram
$$ \begin{diagram} 
\node{M(Q)} \arrow{s} \arrow{e} \node{ {M}(P_-) \times M(P_+)} \arrow{s}  \\
\node{ M_{\rm Ad}(Y)} 
\arrow{e} 
\node{ M_{\rm Ad}(X_-) \times M_{\rm Ad}(X_+) }
\end{diagram} 
 $$
commutes and the top arrow is $\Hom(\pi_1(Y),Z^{\ss})$-equivariant.
The group $\Hom(\pi_1(Y),Z^{\ss})$ acts on the right side of the
diagram via the restriction homomorphism to
$$\Hom(\pi_1(Y),Z^{\ss}) \to \Hom(\pi_1(X_-),Z^{\ss}) \times
\Hom(\pi_1(X_+),Z^{\ss}) $$
Hence $L(Q)$ is contained in the image of the bottom arrow.
\end{remark} 

\begin{proposition}  \label{embed}  {\rm (Sufficient condition for the correspondence
associated to a bordism to be embedded)} Suppose that one of the
  push-forward maps $\pi_1(X_+) \to \pi_1(Y), \ \pi_1(X_-) \to
  \pi_1(Y)$ is surjective.  Then $\rho : M(Q) \to L(Q)$ is a
  bijection.  Furthermore, suppose that
\begin{enumerate} 
\item for any $[\alpha]\in M(P_\pm)$ we have
$$H^0(X_\pm;\d_\alpha^{[\g,\g]}) = \{ 0 \}, \quad
  H^2(X_\pm;\d_\alpha^{[\g,\g]}) = \{ 0 \} .$$
\item for any $[\alpha]\in M(Q)$ we have 
$$H^0(Y;\d_\alpha^{[\g,\g]}) = \{ 0 \}, \quad
  H^2(Y;\d_\alpha^{[\g,\g]}) = \{ 0\} , \quad H^1(Y,\partial Y;
  \d_\alpha^{[\g,\g]}) = \{ 0 \} .$$
\end{enumerate} 
Then $\rho: M(Q) \to L(Q) \subset M(P_-) \times M(P_+)$ is an
embedding.
\end{proposition} 

\begin{proof}   
Suppose that one of $\pi_1(X_\pm) \to \pi_1(Y)$ is surjective.  Then
the maps on the corresponding representation varieties for $\Ad(G)$
and $Z^{\ss}$ in the diagram above are injective.  Since the maps on
the fiber and the base are injective, $\rho$ is injective.  To prove
the second statement, note that the assumptions imply that the moduli
spaces $M(Q),M(P_-), M(P_+)$ are smooth by Lemma \ref{smoothmod}.  The
linearization of restriction $M(Q) \to M(P_-) \times M(P_+)$ is then
restriction on cohomology, 
$$H^1(Y;\d_\alpha^{[\g,\g]}) \to
H^1(X_-;\d_{\alpha_-}^{[\g,\g]}) \times
H^1(X_+;\d_{\alpha_+}^{[\g,\g]}) . $$  
This is injective by the assumption $H^1(Y,\partial Y;
\d_\alpha^{[\g,\g]})= \{ 0 \}$.  Since any injective immersion of a compact
space is an embedding, this completes the proof.
\end{proof}

The notion of composition of bordisms extends naturally to composition
of bundle bordisms.  Let $Q_0 \to Y_0$ be a bundle bordism from $P_0
\to X_0$ to $P_1 \to X_1$ and $Q_1 \to Y_1$ a bundle bordism from $P_1
\to X_1$ to $P_2 \to X_2$, the composition $Q_0 \circ Q_1$ is defined
using equivariant collar neighborhoods of $P_1$ in $Q_0$ and $Q_1$,
and is independent up to bundle isomorphism of the choices.

\begin{remark} {\rm (Correspondences for a composition of bundle
bordisms)} Pullback of connections under $\pi:Q_0 \sqcup Q_1 \to Q_0 \circ Q_1$ induces a map
\begin{equation} \label{fibereq}
\pi^*: M(Q_0 \circ Q_1) \to M(Q_0) \times_{M(P_1)} M(Q_1)
 \end{equation}
where the fiber product is defined using the restriction maps $M(Q_j)
\to M(P_1), j =0,1$.  The map $\pi^*$ fits into a commutative diagram
$$\begin{diagram} \node{ M(Q_0 \circ Q_1)} \arrow{s,l}{\rho}
  \arrow{e,t}{\pi^*} \node{M(Q_0) \times_{M(P_1)} M(Q_1)}
  \arrow{s,r}{{\rm pr}_{M(P_0) \times M(P_2)}\circ (\rho \times \rho)}
  \\ \node{ L(Q_0 \circ Q_1)} \arrow{e,t}{{\rm id}_{M(P_0) \times
      M(P_2)}} \node{L(Q_0) \circ L(Q_1)} \\
\end{diagram}.$$
The bottom map has image in $L(Q_0) \circ L(Q_1)$ by the following
argument:  Let $[\alpha] \in M(Q_0 \circ Q_1)$ 
be a lift of $([\eta_0],[\eta_2]) \in L(Q_0 \circ Q_1) \subset M(P_0)
\times M(P_2)$.  Restricting to $P_1$ gives an element  
$[\eta_1] = [ \alpha |_{P_1} ] \in M(P_1)$ satisfying 
$$ ([\eta_0],[\eta_1])=\rho([\alpha |_{Q_0}]) \in L(Q_0), \quad
([\eta_1],[\eta_2])=\rho([\alpha |_{Q_1}]))\in L(Q_1) .$$
Hence $([\eta_0],[\eta_2])$ lies in
$$ L(Q_0)\circ L(Q_1) =\bigl\{(\zeta_0,\zeta_2)\,|\, \exists
\zeta_1\in M(P_1) : (\zeta_0,\zeta_1)\in L(Q_0), (\zeta_1,\zeta_2)\in
L(Q_1) \bigr\} .$$
This ends the remark. 
\end{remark} 

Proposition~\ref{embed} gives conditions for the left arrow in the
commutative diagram above to be a diffeomorphism; the following
Proposition does the same for the top arrow; and the right arrow will
be discussed in more specific cases in Section~\ref{sec:cobord}.  The
following composition property of the moduli spaces is a nonabelian
version of the Mayer-Vietoris principle.  For simplicity of notation
we restrict to the case of gluing two connected components, although
the same discussion holds e.g.\ for gluing along two boundary
components of the same connected bundle bordism.

\begin{proposition} \label{fiber} {\rm (Conditions for composition of bordisms
to give compositions of correspondences)} Let $Q_0 \to Y_0$ be a
  bundle bordism from $P_0 \to X_0$ to $P_1 \to X_1$ and $Q_1 \to Y_1$
  a bundle bordism from $P_1 \to X_1$ to $P_2 \to X_2$.  Suppose that
  the following conditions hold:
\begin{enumerate}
\item for any $[\alpha]\in M(P_1)$ we have
  $H^0(X_1;\d_\alpha^{[\g,\g]})= \{ 0 \}$, $H^2(X_1;\d_\alpha^{[\g,\g]}) =
  0$;
\item for any $[\alpha]\in M(Q_j)$ for $j=0$ or $j=1$ we have
$$H^0(Y_j;\d_\alpha^{[\g,\g]})= \{ 0 \}, \quad H^2(Y_j;\d_\alpha^{[\g,\g]})
 = \{ 0 \} ,$$
and the restriction map on stabilizers
$$\{g\in\G_0(Q_j)\,|\,
  g^*\alpha=\alpha\} \to \{g\in\G_0(P_1)\,|\,
  g^*\alpha|_{P_1}=\alpha|_{P_1}\}$$ 
is surjective for $j =0$ or $j = 1$;
\item for any $[\alpha]\in M(Q_0 \circ Q_1)$ we have 
$$H^0(Y_0 \circ Y_1;\d_\alpha^{[\g,\g]})= \{ 0 \}, \quad H^2(Y_0 \circ
  Y_1;\d_\alpha^{[\g,\g]})= \{ 0 \} ,$$ 
and the difference of the restriction maps on homology
$$H^1(Y_0;\d_{\alpha}^{[\g,\g]}) \times
  H^1(Y_1;\d_{\alpha}^{[\g,\g]}) \to H^1(X_1;\d_{\alpha
    |X_1}^{[\g,\g]}), (\eta_0,\eta_1) \mapsto
  \eta_0|_{X_1}-\eta_1|_{X_1}$$ 
is surjective;
\item the restriction map $\rho: \G_0(Q_j) \to \G_0(P_1)$ is
  surjective for $j=0$ or $j=1$.
\end{enumerate} 
Then all the spaces in \eqref{fibereq} are smooth and the map $\pi^*$ is
a diffeomorphism.
\end{proposition}

\begin{proof} First we check that $\pi^*$ is a bijection.   
To see that $\pi^*$ is injective suppose that $\alpha,\alpha'$ are
connections on $Q_0 \circ Q_1$ mapping to the same element of $M(Q_0)
\times M(Q_1)$ under $\pi^*$.  Then there exists
$$g = (g_0,g_1) \in
\G_0(Q_0) \times \G_0(Q_1), \quad g^*( \pi^* \alpha) = \pi^*
\alpha' .$$  
Since $ (g_1 |_{P_1})^* \alpha |_{P_1} = (g_0 |_{P_1})^* \alpha
|_{P_1}$ the product 
$(g_1 |_{ P_1})^{-1} (g_0 |_{P_1})$ is an
automorphism of $\alpha |_{P_1}$:
\begin{equation} \label{auto} (g_1 |_{ P_1})^{-1} (g_0 |_{P_1}) (\alpha |_{P_1}) = \alpha |_{P_1}
.\end{equation}
By assumption (b), we may assume without loss of generality that
\eqref{auto} extends to an automorphism $h$ of $ \pi^* \alpha
|_{Q_1}$.  Then $g_0 |_{P_1} = h g_1 |_{P_1}$ and so $g_0$ and $hg_1$
glue together to a gauge transformation $\tilde{g}$ of $Q_0 \circ Q_1$
of class $W^{1,p}$ for any $2 < p < \infty$.  Since $\tilde{g}^*
\alpha = \alpha'$, and $\alpha, \alpha'$ are smooth, $\tilde{g}$ is
also smooth.

To see that $\pi^*$ is surjective, let $(\alpha_0,\alpha_1)$ represent
an element of $M(Q_0) \times_{M(P_1)} M(Q_1)$.  The restrictions
$\alpha_0 |_{P_1}, \alpha_1 |_{P_1}$ are gauge equivalent by some $g_1
\in \G(P_1)$.  By assumption (d), we may assume without loss of
generality that $g_1$ extends over $Q_1$, and so after gauge
transformation of $(\alpha_0,\alpha_1)$ that $\alpha_0
|_{P_1}=\alpha_1 |_{P_1}$.  We may also suppose that the determinant
connection $(\delta_0,\delta_1)$ is the pull-back $\pi^* \delta$ of a
smooth determinant connection $\delta$ on $Q_0 \circ Q_1$.  Then,
after another gauge transformation on $Q_1$, we may assume that the
normal components of $\alpha_0$ and $\alpha_1$ agree on $P_1$.  Then
the curvature equation implies that $(\alpha_0,\alpha_1) = \pi^*
\alpha$ for some smooth connection $\alpha$ on $Q_0 \circ Q_1$.

Finally we check that $\pi^*$ is a diffeomorphism.  Let $\alpha$
represent an element of $M(Q_0 \circ Q_1)$ and
$(\alpha_0,\alpha_1)=\pi^*\alpha$.  By vanishing of $H^0$ and $H^2$
for $\alpha_0,\alpha_1, \alpha,$ and $\alpha |_{P_1}$, the moduli
spaces $M(Q_0), M(Q_1), M(Q_0 \circ Q_1),$ and $M(P_1)$ are smooth at
the given points.  The fiber product $M(Q_0) \times_{M(P_1)} M(Q_1)$
is smooth because the surjectivity of the difference map in assumption
(c) implies transversality.  To see that the linearization of $\pi^*$
is an isomorphism, note that the Mayer-Vietoris long exact sequence
gives a short exact sequence in first cohomology (since $ H^0(X_1;
\d_{\alpha |_{P_1}})= \{ 0 \} $ and $H^2(Y_0 \circ Y_1;\d_{\alpha}) = \{ 0 \} $)
$$ 0
 \to H^1(Y_0 \circ Y_1;\d_{\alpha}) 
\to H^1( Y_0; \d_{\alpha_0})
\oplus 
 H^1(Y_1; \d_{\alpha_1}) \to 
H^1(X_1; \d_{\alpha |_{P_1}}) \to 
0  .$$
The induced isomorphism 
$$ H^1(Y_0 \circ Y_1;\d_{\alpha}) = 
H^1( Y_0; \d_{\alpha_0})
\times_{H^1(X_1; \d_{\alpha |_{P_1}})} H^1(Y_1; \d_{\alpha_1})
 $$
is equal to the linearization of \eqref{fibereq}, since the latter is
the tangent space to the fiber product.  It follows that
\eqref{fibereq} is a diffeomorphism.
\end{proof}

\subsection{Moduli spaces for surfaces}
\label{centralsurf} 

In this Section we make the first step towards constructing a
symplectic-valued field theory via Theorem~\ref{cerfthm} by
associating to any bundle any compact connected oriented surface a
moduli space of constant curvature connections.

\begin{remark} 
Let $X$ be a compact, connected, oriented surface without boundary and
$P \to X$ a $G$-bundle.  
\begin{enumerate} 
\item {\rm (Symplectic structure on the affine space of connections)}
  The affine space $\A(P)$ carries a canonical weakly symplectic
  $\G(P)$-invariant two-form $\omega$ that induces the Hodge pairing
  \eqref{Hodge pairing}.  At a point $\alpha \in \A(P)$ the two-form
  is given by
\begin{equation} \label{sympform}
\omega_\alpha: \Omega^1(X,P(\g))\times \Omega^1(X,P(\g)) \to \R,
\ \ (a_1,a_2) \mapsto \int_X \langle a_1 \wedge a_2 \rangle
. \end{equation}
Here $ \langle a_1 \wedge a_2 \rangle $ denotes the form in
$\Omega^{2}(X)$ obtained by combining wedge product and the inner
product on $P(\g)$.  Weakly symplectic means weakly non-degenerate and
closed, where 
\begin{enumerate} 
\item weakly non-degenerate means that for every non-zero $a_1 \in
  \Omega^1(X,P(\g))$ there exists an $a_2 \in \Omega^1(X,P(\g))$ so
  that $\omega_\alpha(a_1,a_2) \neq 0$;
\item closed means $ \d \omega = 0$ where the de Rham operator is
  defined on a Sobolev completion as in \cite[Section 5.3]{lang:man}.

\end{enumerate} 
The action of $\G_0(P)$ is Hamiltonian with moment map
\begin{equation} \label{mmap}
\A(P) \mapsto \Omega^0(X,P([\g,\g]))^\dual, \ \ \ \alpha \mapsto
\int_X \lan F_\alpha^{[\g,\g]} \wedge \cdot \ran  \end{equation}
in the sense that 
\begin{enumerate} 
\item $\omega$ is $\G_0(P)$-invariant, 
\item the map \eqref{mmap} is equivariant with respect to the
  coadjoint action of $\G_0(P)$ on $\Omega^0(X,P([\g,\g]))^\dual$, and
\item the infinitesimal action \eqref{inf action} of $ \xi \in
  \Omega^0(X,P([\g,\g]))$ is the Hamiltonian vector field of the
  function obtained from \eqref{mmap} by pairing: For all $ \eta \in
  \Omega^1(X,P(\g)) $, we have
\begin{equation} \label{mommap}
 \omega (  \d_\alpha \xi , \eta) = 
\int_X \lan \d_\alpha \xi \wedge \eta \ran = 
\int_X \lan  \xi \wedge - \d_\alpha \eta \ran = 
- \ddt \Big|_{t = 0} \int_X \lan
F^{[\g,\g]}_{\alpha + t \pi^* \eta} \wedge \xi \ran .\end{equation} 
\end{enumerate} 
\item {\rm (Moduli of connections as a symplectic quotient
  \cite{at:mo})} The connections with fixed determinant $\det(\alpha)
  = \delta$ form a symplectic submanifold $\det^{-1}(\delta)$ of
  $\A(P)$.  Indeed, the tangent space $\Omega^1(X,P([\g,\g])$ of
  $\det^{-1}(\delta)$ is a symplectic subspace of $\Omega^1(X,P(\g))$.
  The action of $\G_0(P)$ preserves $\det^{-1}(\delta)$.  Hence the
  moduli space of fixed determinant central curvature connections
$$M(P) = \{ \alpha \in \A(P) | \det(\alpha) = \delta,
\ F_\alpha^{[\g,\g]} = 0 \} /\G_0(P) = \A_\delta(P)/\G_0(P) $$
can be viewed as a symplectic quotient.  
\end{enumerate} 
\end{remark} 

Using this point of view we establish the existence of a symplectic
structure on $M(P)$.  Recall, see e.g. \cite{ps:lg}, that the {\em
  dual Coxeter number} $c$ associated to simple $G$ is the positive
integer $c := \lan \rho, \alpha_0 \ran + 1$ where $\alpha_0$ is the
highest root and $\rho$ the half-sum of positive roots.  In particular
for $G = SU(r)$ the dual Coxeter number is $r$.

\begin{proposition} \label{monotone2} {\rm (Symplectic nature of the moduli space)} 
Let $X$ be a compact, connected, oriented surface without boundary and
$P \to X$ a $G$-bundle.  Suppose that for any $[\alpha]\in M(P)$ we
have $H^0(X;\d_\alpha^{[\g,\g]}) = \{ 0 \}$.  Then $M(P)$ is a compact
symplectic orbifold of dimension $(2 g(X)-2)\dim G_0$, where $g(X)$ is
the genus of $X$.  If furthermore $[G,G]$ is simple, then $M(P)$ is
monotone with monotonicity constant $1/2c$.
\end{proposition}

\begin{proof}  
Lemma~\ref{pides} proves compactness, and the smoothness assertions were proven in Proposition~\ref{smoothmod} since $H^2(X;\d_\alpha^{[\g,\g]})\cong H^0(X;\d_\alpha^{[\g,\g]})$ by Poincar\'e duality.
The dimension formula follows from Riemann-Roch, 
$$\dim T_{[\alpha]} M(P) = \dim H^1(X;\d_\alpha^{[\g,\g]})= (2
g(X)-2)\dim G_0 .$$
To establish the symplectic structure, note that after Sobolev
completion (taking Sobolev class $L_2^s$ with $s>1$) we may take
$\A(P)$ to be an affine Banach space, equipped with the form $\omega$
above.  The action  
$$\G_0(P) \times \A(P) \to \A(P), \quad (\phi, \alpha) \mapsto (\phi^{-1})^*\alpha $$ 
extends to the Sobolev completion of $\G_0(P)$ of class $L_2^{s+1}$ on
connections of class $L_2^s$.  Hence the moment map equation
\eqref{mommap} holds for $ \xi \in \Omega^0(X,P([\g,\g])$ of class
$L_2^{s+1}$ and $ \eta \in \Omega^1(X,P([\g,\g])) $ of class $L_2^s$.
Recall that $\omega$ is translationally invariant, hence closed in the
sense of forms on Banach manifolds as in \cite[Section 5.3]{lang:man}.
Next, we restrict $\omega$ to $\A_\delta(P)\subset\A(P)$.  The latter
is a Banach submanifold for $s\geq 2$ by Proposition~\ref{smoothmod}.
Since the de Rham operator commutes with pull-back,
$$\omega|_{\A_\delta(P)} \in \Omega^2(\A_\delta(P)), \quad \d \omega|_{\A_\delta(P)} = 0 $$ 
is again a closed two-form.  Finally, note that
$\omega|_{\A_\delta(P)}$ is $\G_0(P)$-invariant and vanishes on the
vertical vectors $\d_\alpha \xi, \xi \in \Omega^0(X,P([\g,\g]))$ by
\eqref{mommap}.  Hence it is the pull-back of a two-form on
the quotient $M(P)$.  The two-form form $\omega_{M(P)}$ on the
quotient is closed because the de Rham operator commutes with
pull-back and the map $\A_\delta(P) \to M(P)$ is a submersion.
Non-degeneracy follows from the identity
\begin{equation} \label{nondeg}
 \int_X \langle a \wedge * a
\rangle = \Vert a \Vert_{L_2}^2, \quad \forall [a] \in H^1(X; \d_{\alpha}) \cong
T_{[\alpha]} M(P) \end{equation} 
for a harmonic representative $a$ of the class $[a]$.  The assertion
on monotonicity is a variation on the Drezet-Narasimhan theorem
\cite[Theorem F]{dr:pi}, see also \cite{la:li} and \cite{ku:pg}.  A
symplectic proof is given in \cite{me:can}.  There it is shown, using
the description in Remark \ref{repdes} below, that the first Chern
class $c_1(M(P))$ is $2c$ times the symplectic class
$[\omega_{M(P)}]$.
\end{proof}  

Now restricting to $G=U(r)$ we have the following main result of this
section.

\begin{theorem} \label{MP}  {\rm (Properties of moduli spaces of connections on a surface)} 
Let $P$ be a principal $U(r)$-bundle of degree $d$ coprime to $r$ over
a compact, connected, oriented surface $X$ without boundary.  
\begin{enumerate} 
\item \label{highg}
If $X$ has genus $g(X)\geq 1$, then the moduli space $M(P)$ is a
  nonempty compact symplectic manifold of dimension 
$$ \dim(M(P)) = (2 g(X)-2)(r^2-1) $$
with even Chern numbers
$$ \lan c_1(M(P)), H_2(M(P)) \ran \subset 2 \Z $$
and satisfying the monotonicity relation 
$$ c_1(M(P)) = 2r [\omega_{M(P)}] .$$
\item If $X$ has genus $g(X)= 0$, then $M(P)=\emptyset$.

\item If $X$ has genus $g(X)=1$, then $M(P)=\pt$ is a point.
\end{enumerate} 
Moreover, $M(P)$ is always connected and simply-connected.
\end{theorem} 

\begin{remark} \label{MPr} {\rm (Non-coprime case)} 
Suppose we are in the situation of Theorem~\ref{MP}, except that $r,d$
do have a common divisor. Then for $g(X)\geq 1$ there exist
connections $\alpha\in\A_\delta(P)$ with non-discrete automorphism
groups, i.e.\ $H^0(X;\d_\alpha^{[\g,\g]})\neq 0$.  For $g(X)=0$ we have
$M(P)=\emptyset$ iff $d/r\in\Z$, and otherwise $M(P)$ is a point.
\end{remark}

\begin{proof}[Proof of Theorem \ref{MP} and Remark \ref{MPr}]
The smoothness assertions are standard, see for example \cite{ns:st},
but we give a proof for convenience.  The smooth manifold structure
follows from Proposition~\ref{smoothmod} and vanishing
$H^0(X;\d_\alpha^{[\g,\g]})$ (and hence of $H^2(X;\d_\alpha^{[\g,\g]})$
by Poincar\'e duality) for all $\alpha\in\A_\delta(P)$. To show the
latter, suppose that $\alpha \in \A_\delta(P)$ is fixed by an
infinitesimal gauge transformation $\xi \in \Omega^0(X,P(\g))$ with
trivial determinant, that is, $ \d_\alpha \xi = 0 .$ This
infinitesimal automorphism acts on the bundle $E := P(\C^r)$ with at
least two distinct eigenvalues.  A choice of group of eigenvalues into
two distinct groups defines a splitting of the corresponding vector
bundle
\begin{equation} \label{Esplit} 
E = E_1 \oplus E_2, \quad \rank(E_1) = k, \quad \rank(E_2) = r-k,
\quad 0 < k < r .\end{equation} 
The adjoint bundle $E^\dual\otimes E$ contains the direct summands
$E_j^\dual\otimes E_j$ and the identity section ${\rm Id}_{E}={\rm
  Id}_{E_1}\oplus{\rm Id}_{E_2}$ is the sum of identities in the
summands.  The curvature also splits $F_\alpha =
F_{\alpha|_{E_1}}\oplus F_{\alpha|_{E_2}}$ and is a multiple of the
identity $F_\alpha={\rm Id}_E \eta$ for some
$\eta\in\Omega^2(X)$. This implies $F_{\alpha|_{E_1}}= {\rm Id}_{E_1}
\eta$.  Hence the first Chern number of $E_1$ equals
$$
\langle c_1(E_1), [X] \rangle 
= \frac {{\rm rk}(E_1)}{r} \langle c_1(E) , [X] \rangle = 
\frac {kd}{r} .
$$ 
Since the first Chern number of $E_1$ should be an integer and
$k<r$, this implies that $r$ and $d$ cannot be coprime.  That is, if
$r,d$ are coprime then no such splitting and hence no infinitesimal
automorphism can occur.  

Now the symplectic structure and monotonicity follow from
Proposition~\ref{monotone2} together with dual Coxeter number $c=r$
for $U(r)$.  The assertion on the minimal Chern number is proved in
\cite{dr:pi}.  The claims on (non)emptiness, connectedness and
simply-connectedness follow from the stratification of $\A(P)$
established in \cite{at:mo}, see especially \cite[Theorem 9.12]{at:mo}
with $G_0=SU(r)$ simply-connected.  If $X\cong T^2$ has genus $g= 1$,
then $\dim M(P)=0$ by the dimension formula and connectedness implies
that $M(P)$ is a point.  See also \cite{sc:mo,bo:al} where moduli
spaces of bundles on elliptic curves are investigated more
extensively.

Alternatively, the (non)emptiness for general $r,d$ can be seen from
the description in Remark~\ref{repdes} below:
\begin{itemize} 
\item If $X$ has genus zero, then $d/r\notin\Z$ implies
  $M(P)=\emptyset$.  Indeed the product of commutators in \ref{repdes}
  must be $z = \exp(2\pi i d/r){\rm Id}_{U(r)}\neq{\rm Id}_{U(r)}$
  which is impossible. On the other hand, $M(P)$ is a single point for
  $d/r\in\Z$.  
\item If $X$ has positive genus $g\geq 1$, then $M(P)$ is nonempty by
  e.g.\ Got\^o's theorem \cite{go:cm}.
\end{itemize} 

Finally, if $r = r' m$ and $d = d'm$ have a common factor $m>1$ and
$g(X)\geq 1$, then by the above the moduli space $M(P')$ for a
$U(r')$-bundle $P'$ of degree $d'$ is nonempty. Any representative of
$[\alpha']\in M(P')$ induces a connection on the corresponding vector
bundle $E'=P'(\C^{r'})$. Then $\alpha'$ also induces a central
curvature connection on $E=\oplus_{i=1}^m E'$.  From this one obtains
a connection on the corresponding principal $U(r)$-bundle, which is
isomorphic to $P$.  The splitting of the bundle implies that the
connection has a non-discrete automorphism group.
\end{proof} 

\begin{remark} {\rm (Holomorphic description of the moduli spaces)} 
The moduli spaces $M(P)$ have a holomorphic description, for $G=U(r)$
due to a famous theorem of Narasimhan-Seshadri \cite{ns:st},
generalized to arbitrary groups in Ramanathan's thesis \cite{ra:th},
\cite{ra:th2}.  However, we never use the holomorphic description.
\end{remark} 

\begin{lemma} \label{Pdontmatter}  {\rm (Symplectomorphisms of moduli spaces
induced by bundle isomorphisms)}  
Let $r,d$ be coprime integers.  
\begin{enumerate}
\item For any two $U(r)$-bundles $P_0 \to X_0, P_1 \to X_1$ of degree
  $d$, any bundle isomorphism $\psi : P_0 \to P_1$ covering an
  orientation-preserving diffeomorphism from $X_0$ to $X_1$ induces a
  symplectomorphism $\psi^*: M(P_1) \to M(P_0)$ given by pull-back of
  representatives.
\item Any two such bundle isomorphisms
from $P_0$ to $P_1$ covering the same diffeomorphism from $X_0$ to
$X_1$ induce the same symplectomorphism.  
\item If $P_0 \to X, P_1 \to X$ are $U(r)$-bundles of degree $d$ over
  the same surface, then the moduli spaces $M(P_0), M(P_1)$ are
  canonically symplectomorphic.
\end{enumerate} 
\end{lemma}

\begin{proof} For manifolds of dimension at most three, bundles are
  classified up to isomorphism by their first Chern class.  For
  connected surfaces, the Chern class is determined by the degree.
  Hence if $\pi_j: P_j \to X_j, j =0,1$ are bundles of the same degree
  then there exists an isomorphism $\phi: P_0 \to P_1$ with $\pi_1
  \circ \phi = \pi_0$.  This induces a map from $\A(P_1)$ to
  $\A(P_0)$, given on the level of tangent spaces by the pull-back of
  basic forms $\phi^* : \Omega^1(P_1,\g)_{\on{basic}} \to
  \Omega^1(P_0,\g)_{\on{basic}} $.  This pull-back acts
  symplectically: for $a,b \in \Omega^1(P_1,\g)_{\on{basic}}$ we have
$$ \int_{X_1} \pi_{1,*} \lan a \wedge b \ran = \int_{X_0} \pi_{0,*}
  \lan \phi^* a \wedge \phi^* b \ran $$
since $\phi$ covers an orientation-preserving diffeomorphism from
$X_0$ to $X_1$.  Moreover pull-back by $\phi$ maps $\A_\delta(P_1)$ to
$\A_{\phi^* \delta}(P_0)$ and is equivariant with respect to the gauge
actions of $\G_0(P_1)$ and $\G_0(P_0)=\phi^*\G_0(P_1)$.  Hence
$\phi^*$ descends to a symplectomorphism $M(P_1) \to M(P_0)$.  Any two
such isomorphisms covering the same diffeomorphism differ by a gauge
transformation of $P_1$.  Since any gauge transformation of $P_1$
induces the identity on $M(P_1)$ and $M(P_1)$ is independent of the
choice of determinant connection (see Remark~\ref{delta dont matter}),
the identification $M(P_1) \to M(P_0)$ depends only on the choice of
diffeomorphism from $X_0$ to $X_1$.
\end{proof}  

\begin{remark} \label{repdes} {\rm (Moduli spaces as representations of fundamental group of the punctured
surface)} Atiyah-Bott \cite{at:mo} provide a description of $M(P)$ in
  terms of representations, which makes Lemma~\ref{pides} more precise
  in the case of bundles over surfaces.  In that case, $M(P)$ can be
  identified with the moduli space of $G_0$-representations of the
  fundamental group of the punctured surface $X - \{ x \}$, whose
  value on a small loop around $x$ is equal to a certain central
  element $z \in G_0$.  Here $z$ is determined by the choice of the
  bundle $P \to X$.  For example in the case that $G = U(r)$ and $\lan
  c_1(P),[X] \ran = d$ we have $z =\exp(2\pi i d/r){\rm Id}_{U(r)}$.
  More explicitly, the fundamental group of $X - \{ x \}$ can be
  described as the free group on $2g$ generators
  $\alpha_1,\beta_1,\ldots,\alpha_g,\beta_g$ with the product of
  commutators $ \prod_{j=1}^g [\alpha_j,\beta_j] $ being the class of
  the loop around the base point $x$.  Thus we have a homeomorphism
\begin{equation} \label{reps} {M}(P) \cong \biggl\{
    (a_1,b_1,\ldots,a_g,b_g) \in G_0^{2g} \ \ \bigg| \ \ \prod_{j=1}^g [a_j,b_j]
  = z \biggr\}/ G_0 .\end{equation}
In particular, for $X\cong S^2$ we have 
$$ M(P)  =
\begin{cases}  \emptyset & z \neq{\rm Id}_G \\ 
               \on{pt}  & z={\rm Id}_G . \end{cases}$$  
The map \eqref{reps} is given by choosing a determinant connection
whose curvature is concentrated near a base point, and mapping each
$[\alpha]$ to the representation of $\pi_1(X - \{ x \})$ given by the
holonomies.
\end{remark}  

\subsection{Moduli spaces for three-dimensional elementary bordisms}
\label{sec:cobord}

In this Section we make the second step towards constructing a
symplectic-valued field theory in dimension $2+1$ via
Theorem~\ref{cerfthm}, by defining a partial functor on elementary
bordisms.  

\begin{definition} {\rm (Correspondences for bordisms)}  
Fix $G = U(r)$ for some positive integer $r$ and let $P_\pm \to X_\pm$
be principal $U(r)$-bundles of the same degree $d$ coprime to $r$ over
surfaces $X_\pm$ as in Section~\ref{centralsurf}.  Let $(Q,\psi)$ be a
bundle bordism from $P_- \to X_-$ to $P_+ \to X_+$ as in
Section~\ref{subsec mod cob}; in particular $Y$ is a $3$-dimensional
bordism from $X_-$ to $X_+$.  The image of $M(Q)$ under pullback as in
Definition \ref{corresp} is
$$ L(Q) := \rho(M(Q)) \subset M(P_-)^- \times M(P_+) . $$
\end{definition} 

\begin{remark} {\rm (Low genus cases)}  
\begin{enumerate} 
\item If either $X_+$ or $X_-$ has genus $0$, then the corresponding
  moduli space $M(P_\pm)$ is empty by Theorem \ref{MP}, and therefore
  so is $L(Q)$.
\item 
 If either $X_+$ or $X_- = T^2$ has genus $1$, then the corresponding
 moduli space $M(P_\pm)$ is a point.  Furthermore, by a special case
 of the following Theorem, $L(Q)$ is a Lagrangian correspondence from
 a point to $M(P_\mp)$ and can also be viewed as Lagrangian
 submanifold of $M(P_\mp)$.
\end{enumerate}
\end{remark}

\begin{theorem} \label{thm lagembed} {\rm (Lagrangians for elementary bordisms)}  
If $Y$ is a compression body as in Definition~\ref{compbody} and $Q
\to Y$ is a principal $G$-bundle then $ \rho:M(Q) \to L(Q) \subset
M(P_-)^- \times M(P_+)$ is a Lagrangian embedding.  If moreover $Y$ is
an elementary bordism as in Definition \ref{Morse datum}
\eqref{elementary}, then $L(Q)$ is simply-connected and spin.
\end{theorem}

\begin{proof} 
By Theorem~\ref{MP}, $H^0(X_\pm;\d_\alpha^{[\g,\g]}) \cong
H^2(X_\pm;\d_\alpha^{[\g,\g]}) = \{ 0 \}$ for any $[\alpha]\in M(P_\pm)$.
Lemma \ref{relative}~(a) below shows that the assumptions of
Proposition \ref{embed} are satisfied so that $\rho$ is an embedding.
The image is Lagrangian since the image of $T_{[\alpha]} {M}(Q)\cong
H^1(Y;\d_\alpha) $ in $T_{([\alpha |_{P_-}],[ \alpha|_{P_+}])}
({M}(P_-)^- \times M(P_+)) \cong H^1(\partial Y;\d_{\alpha|_{\partial
    Y}})$ is a Lagrangian subspace by Corollary \ref{c maxiso}.

Next, suppose that $Y$ is elementary and without loss of generality
$X_-$ is the surface of lower genus. By Lemma \ref{holL} below, $L(Q)$
is a principal $G_0$-bundle over the moduli space $M(P_-)$.  The space \label{change10}
$M(P_-)$ is simply-connected by Theorem~\ref{MP}.  Since the fiber and
base are simply-connected, so is $L(Q)$.  Now let $\pi$ denote the
projection of $L(Q)$ onto $M(P_-)$.  By Lemma \ref{holL} the tangent
bundle of $L(Q)$ is the sum of the pull-back of $TM (P_-)$ with a
vector bundle with fiber $\g_0$ and structure group $G_0$.  By Lemma
\ref{monotone2} the canonical bundle of $M(P_-)$ admits a square root.
Since $G_0$ is simply-connected, the representation $\g_0$ is spin and
hence so is the vertical part of $TL(Q)$.  Hence $L(Q)$ is also spin.
\end{proof}

\begin{lemma} \label{relative}  
Suppose that $Y$ is a compression body as in Definition~\ref{compbody}
between surfaces $X_-,X_+$ of positive genus, and the genus of $X_+$
is larger than or equal to the genus of $X_-$. Then the following
holds.
\begin{enumerate} 
\item[(a)] The map of fundamental groups $\pi_1( X_+ ) \to \pi_1(Y)$ is a
  surjection and the map $\pi_1( X_- ) \to \pi_1(Y)$ is an injection.
\item[(b1)] For any connection $\alpha \in \A_\delta(Q)$, the map
  induced by restriction $ H^0(Y,\d_\alpha^{[\g,\g]}) \to
  H^0(X_+,\d_{\alpha |_{P_+}}^{[\g,\g]})$ is an isomorphism and the
  map $ H^0(Y,\d_\alpha^{[\g,\g]}) \to H^0(X_-,\d_{\alpha
    |_{P_-}}^{[\g,\g]})$ is an injection;
\item[(b2)] the map
  $H^1(Y,\d_\alpha^{[\g,\g]}) \to H^1(X_+,\d_{\alpha
    |_{P_+}}^{[\g,\g]})$ is injective and the map $
  H^1(Y,\d_\alpha^{[\g,\g]}) \to H^1(X_-,\d_{\alpha
    |_{P_-}}^{[\g,\g]})$ is surjective;
\item[(b3)] the map $ H^2(Y,\d_\alpha^{[\g,\g]}) \to
  H^2(X_-,\d_{\alpha |_{P_-}}^{[\g,\g]} )$ is an isomorphism and the
  map $ H^2(Y,\d_\alpha^{[\g,\g]}) \to H^2(X_+,\d_{\alpha
    |_{P_+}}^{[\g,\g]} )$ is a surjection.
\end{enumerate}
\end{lemma}

\begin{proof} 
  For simplicity of notation assume that
  $\partial Y = \ol{X}_- \sqcup X_+$, that is, the boundary
  identification $\phi$ is the identity.  By assumption, $Y$ supports
  a Morse function that is maximal on $X_+$, minimal on $X_-$, and
  only has critical points of index $1$. Then $Y$ deformation retracts
  onto the union of $X_-$ and the stable manifolds of the critical
  points, which are intervals.  Since $X_-$ is connected, the
  attaching points of these intervals may be homotoped so that they
  are equal. Thus $Y$ is homotopic to the wedge product of $X_-$ and a
  collection of circles.  By the Seifert-van Kampen theorem,
  $\pi_1(Y)$ is the free product of $\pi_1(X_-)$ and a copy of $\Z$
  for each critical point.  Hence $\pi_1(X_-) \to \pi_1(Y)$ is an
  injection.  Similarly $Y$ is obtained from $X_+$ by attaching
  $2$-handles.  Seifert-van Kampen in this case presents $\pi_1(Y)$ as
  the quotient of $\pi_1(X_+)$ by actions of $\Z$ for each critical
  point, given by the image of the attaching cycle in $\pi_1(X_+)$.

The assertions on zeroth cohomology follow from its interpretation as
infinitesimal automorphisms of a flat connection, which are the Lie
algebras of the centralizers of its holonomy groups, and part (a).
The remaining assertions follow from the Mayer-Vietoris principle for
the cohomology of $\d_\alpha$.  To apply it, we consider $Y$ as the
union $U \cup V$ where $U$ is a neighborhood of $X_\pm$ and $V$ is a
neighborhood of the union of stable manifolds of the critical points,
so that $U \cap V$ is a neighborhood of the attaching cycles. Then
$H^0(V;\d_\alpha^{[\g,\g]}) \to H^0(X_+ \cap V;\d_\alpha^{[\g,\g]})$
is surjective, since any infinitesimal automorphism of a flat
connection over the boundary of a disk extends over the interior.  By
the long exact sequence $H^1(Y;\d_\alpha^{[\g,\g]}) \to
H^1(X_+;\d_{\alpha | P_+}^{[\g,\g]})$ is injective.  On the other hand
$H^1(X_- \cap V;\d_\alpha^{[\g,\g]})$ vanishes, since $X_- \cap V$ is
homotopic to a finite set of points and the Poincar\'e lemma holds for
cohomology with twisted coefficients.  Hence
$H^1(Y;\d_\alpha^{[\g,\g]};\d_\alpha^{[\g,\g]})$ surjects onto
$H^1(X_-;\d_{\alpha | P_-}^{[\g,\g]})$.  Since $H^1(X_- \cap
V;\d_\alpha^{[\g,\g]})$ and $H^2(X_- \cap V;\d_\alpha^{[\g,\g]})$
vanish, $H^2(Y;\d_\alpha^{[\g,\g]} )$ is isomorphic to
$H^2(X_-;\d_{\alpha | P_-}^{[\g,\g]}) \cong H^0(X_-;\d_{\alpha |
  P_-}^{[\g,\g]})$.  Finally surjectivity onto $H^2(X_+;\d_{\alpha |
  P_+}^{[\g,\g]})$ follows from the exact sequence.
\end{proof}  

In order to understand the topology of $L(Q)$ we use an explicit description in terms of representations of the extended fundamental group as in Remark~\ref{repdes}.

\begin{lemma}
\label{holL}
Suppose that $Y$ is an elementary bundle bordism from a surface $
X_-$ of genus $g$ to a surface $X_+$ of genus $g + 1$. Let $Q$ be a
$G$-bundle over $Y$.  Then in the description of Remark~\ref{repdes},
the Lagrangian $L(Q)\subset M(P_-)^- \times M(P_+)$ is the set of pairs
of equivalence classes
$$ \bigl([a_1,b_1,\ldots,a_g,b_g], 
[a_1,b_1,\ldots,a_{g+1},{\rm Id}]
\bigr) \in G_0^{2g}/G_0 \times G_0^{2g+2}/G_0
$$ 
satisfying the relation $\prod_{j=1}^{g} [a_j,b_j]= z $. In
particular, the projection $\pi:L(Q)\to M(P_-)$ gives $L(Q)$ the
structure of a $G_0$-bundle over $M(P_-)$, whose vertical tangent
space is the associated bundle
$$ \ker(D \pi) \cong \Set{ [a_1,b_1,\ldots,a_g,b_g,\xi] \in G_0^{2g}
\times \g | \prod_{j=1}^g [a_j,b_j] = z  } / G_0
.$$
\end{lemma} 

\begin{proof}
By Lemma \ref{pides}, the moduli space $M(Q)$ is a
$\Hom(\pi_1(Y),Z^{\ss})$ fiber bundle over a component of
$\Hom(\pi_1(Y),\Ad(G))/\Ad(G)$.  Choose generators
$$\alpha_1,\beta_1,\ldots,\alpha_{g+1},\beta_{g+1} \in \pi_1(X_+), 
\quad \prod_{j=1}^{g+1} [\alpha_j,\beta_j]={\rm
  Id}_{\pi_1(X_+)} $$
such that the maps of fundamental groups in Lemma~\ref{relative}~(a)
  realize 
\begin{enumerate} 
\item  $\pi_1(Y)$ as the quotient of $\pi_1(X_+)$ by the subgroup
  generated by $\beta_{g+1}$, and 
\item $\pi_1(X_-)$ as the subgroup of $Y$ generated by
  $\alpha_1,\beta_1,\ldots,\alpha_{g},\beta_{g}$. 
\end{enumerate} 
Now $\Hom(\pi_1(Y),\Ad(G))/\Ad(G)$ and $\Hom(\pi_1(Y),Z)$ are the
subsets of the spaces of representations of $\pi_1(X_+)$ given by
requiring that the representation vanishes on $\beta_{g+1}$.  It
follows that $L(Q)$ is equal to the space of representations mapping
$\beta_{g+1}$ to some element $h$ of $Z$.

Now we assume that $Q$ is equipped with a determinant connection whose
curvature vanishes except on a small neighborhood of a path connecting
base points in $X_\pm$, disjoint from the stable and unstable
manifolds.  In particular, the curvature vanishes on a disk with
boundary $\beta_{g+1}$.  Since $h$ is equal to the holonomy around the
vanishing cycle $\beta_{g+1}$, $h$ must equal the identity.
\end{proof}

This finishes the proof of Theorem~\ref{thm lagembed}.  Together with
the subsequent Lemma this defines the functors for elementary bordisms
for any fixed coprime $r,d$:

\begin{lemma} {\rm (Existence and uniqueness of Lagrangians associated to elementary bordisms)}  For $d,r$ coprime positive integers:
\begin{enumerate} 
\item For any elementary bordism $Y$ from $X_-$ to $X_+$ there
  exists a $U(r)$-bundle $Q \to Y$ such that both $Q|_{X_-}$ and
  $Q|_{X_+}$ have degree $d$.
\item 
The Lagrangian correspondence $L(Q)\subset M(P_-)^-\times M(P_+)$ is
independent of the choice of $Q$ under the canonical
symplectomorphisms of Lemma~\ref{Pdontmatter}.
\item 
The Lagrangian correspondence associated to the trivial bordism
$[0,1]\times X$ is the diagonal $L([0,1]\times P)=\Delta_{M(P)}\subset
M(P)^-\times M(P)$.
\end{enumerate} 
\end{lemma} 

\begin{proof} 
Let $Y$ be an elementary bordism from $X_-$ to $X_+$.  A simple
computation using e.g. cellular homology shows that since $Y$ is
elementary, $H^2(Y) \cong H^2(X_\pm) \cong \Z$.  Taking a $U(r)$
bundle induced from a $U(1)$-bundle with first Chern class $d$ shows
that the bundle $Q$ above exists.  Since the homotopy groups
$\pi_1(SU(r))$ and $\pi_2(SU(r))$ are trivial, the bundle $Q$ is
unique up to isomorphism.  These arguments imply that $M(Q)$ is
well-defined and $L(Q)$ is independent of the choice of $Q$.  In the
case of a product bordism $Q = [0,1 ] \times P$, we take $\delta$ to
be the pull-back of a determinant connection on $P_\pm$.  Since any
connection is gauge-equivalent to one vanishing on the fibers of
$[0,1] \times P \to P$, pull-back defines an isomorphism $M(Q) \cong
M(P_\pm)$.  Thus $M(Q)$ is embedded in $M(P_-) \times M(P_+)$ via the
diagonal.
\end{proof}

\subsection{Cerf moves for moduli spaces}

This Section provides the final steps in the construction of a
symplectic-valued field theory in dimension $2+1$ via
Theorem~\ref{cerfthm}.  As before, we fix $G = U(r)$ for some $r\in\N$
and fix a degree $d\in\Z$ coprime to $r$.

\begin{theorem} \label{sympval} {\rm (Field theory via moduli spaces)}   
For $d,r$ coprime positive integers there exists a unique
$2+1$-dimensional connected symplectic-valued field theory 
$ \Phi: \Cob_{2+1}^0 \to \Symp^{\sharp}$
that associates
\begin{enumerate}
\item 
to each surface $X$ the moduli space 
$$ \Phi(X)=M(P)=:M(X)$$
constructed in Section~\ref{centralsurf} using a choice of
$U(r)$-bundle $P \to X$ of degree $d$, and
\item
to each elementary bordism $[Y]$ from $X_-$ to $X_+$, the equivalence class 
$$\Phi(Y)=[L(Y)], \quad 
L(Y):=L(Q)\subset M(X_-)^-\times M(X_+)$$
constructed in Section~\ref{sec:cobord} using a choice of
$U(r)$-bundle $Q \to Y$ pulling back
to the given bundles over $X_\pm$.
\end{enumerate} 
\end{theorem} 

\begin{remark}  {\rm (Generalized correspondences for bordisms)}  
Theorem \ref{sympval} associates to any morphism $[Y]$ from $X_-$ to
$X_+$ in $\Cob_{2+1}^0$ an equivalence class of generalized Lagrangian
correspondences as follows. Let
$$ Y = Y_{1} \cup_{X_1} Y_{2} \cup_{X_2} \ldots \cup_{X_{r-1}} Y_m $$
be a Cerf decomposition into elementary bordisms.  Associated to each
piece $Y_{k}$ is a Lagrangian correspondence $L(Y_k)$ from
$M(X_{k-1})$ to $M(X_k)$, where by convention $M(X_0) = M(X_-)$ and
$M(X_m) = M(X_+)$.  Hence our construction associates a generalized
Lagrangian correspondence from $M(X_-)$ to $M(X_+)$ to the Cerf
decomposition of $Y$,
$$ 
\ul{L}((Y_k)_{k=1,\ldots,m}) := \bigl(L(Y_{1}), \ldots, L(Y_m)\bigr) .
$$
Theorem \ref{sympval} implies that the equivalence class 
$$ \Phi(Y) = [\ul{L}((Y_k)_{k=1,\ldots,m}))] = [(L(Y_{1}), \ldots,
  L(Y_m))] = \Phi(Y_{1})\circ \ldots \circ\Phi(Y_m) $$
  as an element in $\Hom_{\Symp^{\sharp}}( M(X_-),M(X_+))$
  is independent of the choice of Cerf decomposition. \label{change12}
  This ends the remark.
\end{remark} 

We already proved in Theorem \ref{thm lagembed} that the Lagrangian
correspondences involved are simply-connected and spin.  To complete
the proof of the theorem it suffices to show that the Cerf relations
of Theorem \ref{cerfthm} are satisfied for the Lagrangian
correspondences constructed in the previous section.  

\begin{lemma} 
For each of the Cerf moves in Theorem \ref{cerf}, the conditions in
Proposition \ref{fiber} hold.
\end{lemma} 

\begin{proof}
 Note that by assumption $H^0$ and $H^2$ of $\d_\alpha^{[\g,\g]}$
 vanish on $Y_i, Y_{i+1},X_{i-1},X_i,X_{i+1}$.  Recall $G_0 = [G,G] =
 SU(r)$ and $\G_0(P)$ is the corresponding group of gauge
 transformations.  The restriction map $\G_0(Q) \to \G_0(P_-) \times
 \G_0(P_+)$ is surjective, since $\pi_1(G_0), \pi_2(G_0)$, vanish.  It
 follows that there is no obstruction to extending over the $1,2$ and
 $3$-dimensional cells of the bordism.  It remains to check the
 surjectivity of the difference of the restriction maps $H^1(Y_i
 \sqcup Y_{i+1};\d_{\alpha}^{[\g,\g]}) \to H^1(X_i;\d_{ \alpha |
   P_i}^{[\g,\g]})$.  By Mayer-Vietoris this surjectivity is
 equivalent to injectivity of $H^2(Y_i \cup Y_{i+1};
 \d_\alpha^{[\g,\g]}) \to H^2(Y_i \sqcup Y_{i+1};
 \d_\alpha^{[\g,\g]})$.  This injectivity holds if $H^2(Y_i \cup
 Y_{i+1};\d_\alpha^{[\g,\g]}) = \{ 0 \}$.  We check this in each case.

\vskip .1in \noindent {\em Critical point cancellation}: Suppose that
two adjacent pieces $Y_i,Y_{i+1}$ are replaced by a single piece $Y_i
\cup Y_{i+1}$ diffeomorphic to a cylinder $[-1,1] \times X_{i-1}$.
Then every point in $M(Y_i \cup Y_{i+1}) \cong M(X_{i-1})$ has
vanishing $H^2$, as required.

\vskip .1in \noindent 
{\em Gluing in a cylinder}: An elementary bordism $Y_i$ and
cylindrical bordism $Y_{i+1}$ is replaced with another elementary
bordism diffeomorphic to $Y_i$.  Then every point in $M(Y_i \cup
Y_{i+1}) \cong M(Y_i)$ has vanishing $H^2$, as required.

\vskip .1in \noindent 
{\em Reversing order of critical points:} This move can be broken down
into stages, where in the first stage two elementary bordisms are
replaced by a compression body, and in the second stage the
compression body is replaced by two other elementary bordisms.  If
the two critical points have the same index, then $H^2$ vanishes by
Lemma \ref{relative}.  In the case of differing index, suppose that
the two-handle is attached first. Then the surjectivity claim holds by
Lemma \ref{relative}, since $X_i$ has larger genus than $X_{i+1}$ and
$Y_i$ is a compression body.  Thus $H^2$ vanishes.  By Mayer-Vietoris
the surjectivity property in Proposition \ref{fiber} holds for the
decomposition corresponding to attaching a one-handle first, as well.
\end{proof} 

\begin{remark} {\rm (Holonomy description of Cerf moves)}  
Suppose that $ X_{i-1}$ is a surface of genus $g$ and $X_i$ is a
surface of genus $g + 1$.
\begin{enumerate}
\item {\rm (Critical point cancellation)} By Remark \ref{holL}
  $L(Y_i)$ resp. the transpose of $L(Y_{i+1})$ may be identified with
  the set of pairs of orbits
$$ ([a_1,b_1,\ldots,a_g,b_g], [a_1',b_1',\ldots,a_{g+1}',b_{g+1}']) $$
such that
$$ b_{g+1}' = \Id, \ \ a_j = a_j', \ \ b_j = b_j', \ \ j=1,\ldots,g $$
resp.
$$ a_{g+1}' = \Id, \ \  a_j = a_j', \ \ b_j = b_j', \ \ j=1,\ldots,g .$$
Thus the composition $ L(Y_i) \circ L(Y_{i+1})$ is the diagonal. 
\item {\rm (Critical point switch)} Suppose that the initial
  decomposition $Y_i,Y_{i+1}$ corresponds to attaching a one-handle
  and then attaching a two-handle to $X_{i-1}$.  The surface $X_i$ is
  obtained from $X_{i+1}'$ by attaching two one-handles, so that the
  attaching one-cycles in $X_i$ correspond to disjoint generators of
  $\pi_1(X_i)$.  These generators we may take to equal $\alpha_1$
  resp. $\alpha_{g+1}$.  Then
$$L(Y_i) = \left\{ \begin{array}{l} (
    [a_2,b_2,\ldots,a_{g},b_{g}],[a_1',b_1',\ldots,a_{g+1}',b_{g+1}']
    )  \\ a_{1} = \Id, \ \ a_j = a_j', \ \ b_j = b_j',
    \ \ j=2,\ldots,g \end{array} \right\} .$$
Furthermore 
$$L(Y_{i+1})
= \left\{
\begin{array}{l} 
 ([a_1',b_1',\ldots,a_{g+1}',b_{g+1}'],[a_1'',b_1'',\ldots,a_{g}'',b_{g}'']) \\ 
 a_{g+1}' = \Id, \ \ a_j' = a_j'', \ \ b_j' = b_j'', \ \ j=1,\ldots,g \end{array}  \right\}.
$$
Thus $L(Y_i \cup Y_{i+1})$ is embedded into $M(X_{i-1})^- \times
M(X_{i+1})$, and the composition $L(Y_i) \circ L(Y_{i+1})$ is
transversal and equal to $L(Y_i \cup Y_{i+1})$.  The arguments for the
other order of attaching, and the other indices of critical points, are
similar.
\end{enumerate} 
\end{remark}

Applying Proposition \ref{fiber} and Theorem \ref{cerfthm} yields
Theorem \ref{sympval}.

\section{Functors for bordisms via quilts}
\label{floer for 3}

In this section, we associate to any three-dimensional bordism as
above a functor between the extended Fukaya categories of the moduli
spaces of connections associated to the incoming and outgoing
surfaces, using quilted Floer theory.

\subsection{Quilted Floer theory} 

In the paper \cite{quiltfloer} we generalized Lagrangian Floer theory
to the setting of Lagrangian correspondences.  

\begin{definition} {\rm (Quilted Floer homology)}  Suppose that
\begin{enumerate} 
\item $\ul{M} = (M_0,M_1,\ldots,M_m)$ is a sequence of compact
  simply-connected monotone symplectic manifolds with the same
  monotonicity constant $\tau$ and $N$-fold Maslov covers
  $\Lag^N(M_k) \to \Lag(M_k), k = 0,\ldots, m$;
\item $\ul{L} = (L_{01},L_{12},\ldots, L_{(m-1)m})$ is a sequence of
  compact graded relatively spin Lagrangian correspondences from $M_0$
  to $M_1$, $M_1$ to $M_2$ etc.
\item $L_0 \subset M_0, L_m \subset M_m$ are compact simply-connected
  graded relatively spin Lagrangian submanifolds.
\end{enumerate} 
The {\em quilted Floer cohomology} $HF(\ul{L})$ of a generalized
Lagrangian correspondence $\ul{L} = (L_1,L_2,\ldots, L_m)$ is the
$\Z_N$-graded cohomology of a complex whose differential is a signed
count of {\em quilted} holomorphic strips with boundary in $L_0,L_m$
and seams in $L_{(j-1)j}, j = 1,\ldots, m$.
\end{definition} 

\begin{remark} {\rm (Effect of geometric composition on quilted Floer cohomology)}  
We proved in \cite{ww:isom} that the quilted Floer cohomology groups
behave well under composition: if for some $j$, the composition
$L_{(j-1)(j+1)} := L_{(j-1)j} \circ L_{j(j+1)}$ is smooth and embedded
into $M_{j-1}^- \times M_{j+1}$ then the quilted Floer cohomology
group is unchanged up to isomorphism by replacing the pair
$L_{(j-1)j},L_{j(j+1)}$ with $L_{(j-1)(j+1)}$. An example is shown in
Figure \ref{shrinkfig}, for the case $m = 2$.
\end{remark} 

\begin{figure}[ht]
\begin{picture}(0,0)%
\includegraphics{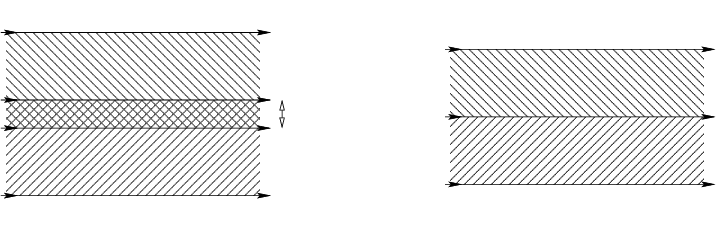}%
\end{picture}%
\setlength{\unitlength}{2200sp}%
\begingroup\makeatletter\ifx\SetFigFont\undefined%
\gdef\SetFigFont#1#2#3#4#5{%
  \reset@font\fontsize{#1}{#2pt}%
  \fontfamily{#3}\fontseries{#4}\fontshape{#5}%
  \selectfont}%
\fi\endgroup%
\begin{picture}(9549,3135)(589,-2461)
\put(2476,504){\makebox(0,0)[lb]{{$\mathbf{L_2}$}}}
\put(2476,-700){\makebox(0,0)[lb]{{$\mathbf{L_{12}}$}}}
\put(2476,-1111){\makebox(0,0)[lb]{{$\mathbf{L_{01}}$}}}
\put(2476,-2300){\makebox(0,0)[lb]{{$\mathbf{L_0}$}}}
\put(8751,-2130){\makebox(0,0)[lb]{{$\mathbf{L_0}$}}}
\put(8751,-961){\makebox(0,0)[lb]{{$\mathbf{L_{02}}$}}}
\put(8751,300){\makebox(0,0)[lb]{{$\mathbf{L_2}$}}}
\put(7626,-511){\makebox(0,0)[lb]{{$\mathbf{M_2}$}}}
\put(7626,-1486){\makebox(0,0)[lb]{{$\mathbf{M_0}$}}}
\put(1276,-1600){\makebox(0,0)[lb]{{$\mathbf{M_0}$}}}
\put(1276,-900){\makebox(0,0)[lb]{{$\mathbf{M_1}$}}}
\put(1276,-200){\makebox(0,0)[lb]{{$\mathbf{M_2}$}}}
\put(4750,-850){\makebox(0,0)[lb]{{$\delta$}}}
\end{picture}%
\caption{Shrinking the middle strip} \label{shrinkfig}
\end{figure}
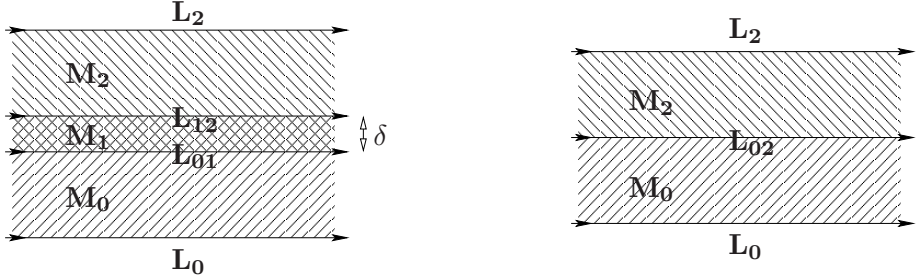

In \cite{Ainfty} we associated to any Lagrangian correspondence
$L_{01}$ from $M_0$ to $M_1$ a functor $\Phi(L_{01})$ from a version
of the Fukaya category of $M_0$ to that of $M_1$:

\begin{definition} \label{genlag} Let $M$ be a $\tau$-admissible
  symplectic manifold as in Definition \ref{admiss} equipped with an
  $N$-fold Maslov cover $\Lag^N(M) \to \Lag(M)$.
\begin{enumerate}
\item {\rm (Generalized Lagrangian submanifolds)} A {\em generalized
  Lagrangian submanifold} of $M$ is a generalized Lagrangian
  correspondence from a point to $M$, that is, a sequence
  $L_{-s(-s+1)}, \ldots, L_{(-1)0}$ of correspondences from $M_{-s} =
  \pt$ to $M_0 = M$.  We say that a generalized Lagrangian
  correspondence satisfies a certain property (simply-connected,
  compact, etc.) if each correspondence in the sequence satisfies that
  property.
\item \label{extfuk} {\rm (extended Fukaya category)} Let $\Fuk^{\sharp}(M)$ denote
  the \ainfty category whose
\begin{enumerate} 
\item objects are compact, oriented generalized Lagrangian
  submanifolds of $M$ equipped with gradings and relative spin
  structures;
\item morphisms from an object $\ul{L}_0$ to an object $\ul{L}_1$ are
  quilted Floer homology cochains:
  \begin{equation} \label{mspace} \Hom(\ul{L}_0,\ul{L}_1) =
    CF(\ul{L}_0,\ul{L}_1) \end{equation}
constructed in \cite{quiltfloer};
\item composition and identities are defined by counting holomorphic
  quilts with strip-like ends and Lagrangian boundary and seam
  conditions as in \cite{Ainfty}.
\end{enumerate}
\item {\rm (Functors for Lagrangian correspondences)} Let $M_0,M_1$ be
  $\tau$-admissible symplectic manifolds.  For any compact, oriented,
  simply-connected relatively-spin graded correspondence
  $L_{01} \subset M_0^- \times M_1$ the functor
$$\Phi(L_{01}): \Fuk^{\sharp}(M_0) \to \Fuk^{\sharp}(M_1)$$ 
is defined on objects by
$$ (L_{-s(-s+1)}, \ldots, L_{(-1)0}) \mapsto (L_{-s(-s+1)}, \ldots,
L_{(-1)0},L_{01}) .$$
On morphisms and at the level of cohomology $\Phi(L_{01})$ is defined
by counting holomorphic quilts of the form in Figure \ref{morphism}.
\end{enumerate} 
\end{definition}  

\begin{figure}[ht]
\begin{picture}(0,0)%
\includegraphics{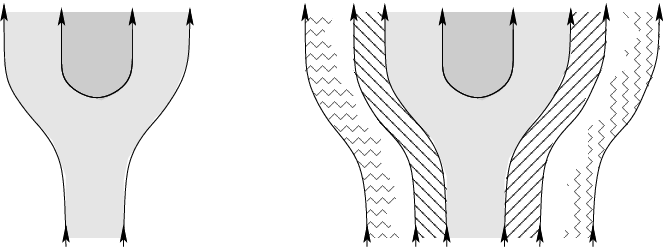}%
\end{picture}%
\setlength{\unitlength}{3108sp}%
\begingroup\makeatletter\ifx\SetFigFont\undefined%
\gdef\SetFigFont#1#2#3#4#5{%
  \reset@font\fontsize{#1}{#2pt}%
  \fontfamily{#3}\fontseries{#4}\fontshape{#5}%
  \selectfont}%
\fi\endgroup%
\begin{picture}(6743,2499)(-580,-3358)
\put(4120,-2536){\makebox(0,0)[lb]{$M_0$}}
\put(4120,-1401){\makebox(0,0)[lb]{$M_1$}}
\put(4120,-1956){\makebox(0,0)[lb]{$L_{01}$}}
\put(3691,-3121){\rotatebox{90.0}{\makebox(0,0)[lb]{{{${L'_{(-2)(-1)}}$}}}}}
\put(3991,-3101){\rotatebox{90.0}{\makebox(0,0)[lb]{{{${L'_{(-1)0}}$}}}}}
\put(4616,-3101){\rotatebox{90.0}{\makebox(0,0)[lb]{{{${L_{(-1)0}}$}}}}}
\put(4956,-3101){\rotatebox{90.0}{\makebox(0,0)[lb]{{{${L_{(-2)(-1)}}$}}}}}
\put(5650,-3301){\rotatebox{90.0}{\makebox(0,0)[lb]{{{${L_{(-m)(-m+1)}}$}}}}}
\put(2906,-3401){\rotatebox{90.0}{\makebox(0,0)[lb]{{{${L'_{(-m')(-m'+1)}}$}}}}}
\put(5906,-1886){\rotatebox{90.0}{\makebox(0,0)[lb]{{{${N_{-m+1}}$}}}}}
\put(5300,-1931){\rotatebox{90.0}{\makebox(0,0)[lb]{{{${N_{-1}}$}}}}}
\put(3280,-1876){\rotatebox{90.0}{\makebox(0,0)[lb]{{{${N'_{-1}}$}}}}}
\put(2826,-1976){\rotatebox{90.0}{\makebox(0,0)[lb]{{{${N'_{-m'+1}}$}}}}}
\put(800,-2671){\makebox(0,0)[lb]{$\ul{L}$}}
\put(250,-2536){\makebox(0,0)[lb]{$M_0$}}
\put(250,-1401){\makebox(0,0)[lb]{$M_1$}}
\put(-300,-2671){\makebox(0,0)[lb]{$\ul{L}'$}}
\put(250,-1956){\makebox(0,0)[lb]{$L_{01}$}}
\put(1846,-2131){\makebox(0,0)[lb]{$=$}}
\end{picture}%
\caption{The Lagrangian correspondence functor $\Phi(L_{01})$ on morphisms}
\label{morphism}
\end{figure}
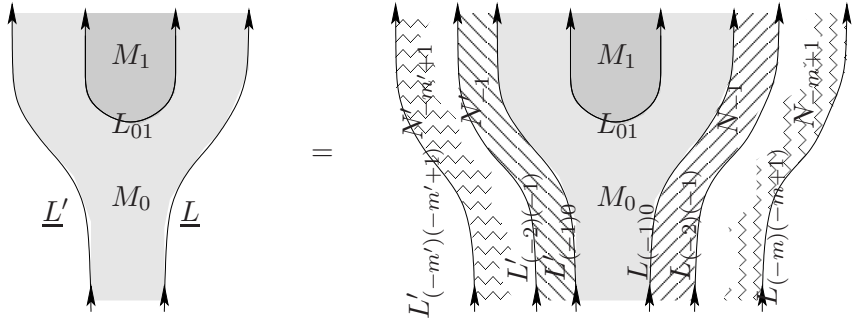

\begin{remark} \label{composeresult} {\rm (Functors for geometric compositions)} 
The composition result on quilted Floer homology generalizes to the
categorical setting as follows \cite{ww:isom}: Suppose that $L_{01}$
and $L_{12}$ are Lagrangian correspondences as above such that $L_{01}
\circ L_{12}$ is smooth and embedded into $M_0^- \times M_2$, and
simply-connected with spin structure induced from that of
$L_{01},L_{12}$ and $M_1$.  Then $\Phi(L_{12}) \circ \Phi(L_{01})$ is
homotopic to $\Phi(L_{01} \circ L_{12})$.  Note that the convention
for composition of functors (first functor on the right) is opposite
to the convention for composition of correspondences (first
correspondence on the left).
\end{remark} 

\begin{definition} \label{ainftycat} {\rm (Category of \ainfty
    categories)} For a fixed positive even integer $N$ let
  $A_\infty(N)$ be the category whose objects are $\Z_N$-graded
  \ainfty categories $\cC$ and whose morphisms are homotopy classes of
  \ainfty functors $\Phi: \cC_0 \to \cC_1$.  Composition is
  composition of functors.  Because our theory is only relatively
  graded, we also require that the categories are equipped with {\em
    shift functors}
\[ S: \cC \to \cC, \quad \Hom(S(L_0), L_1)[1] = \Hom(L_0,L_1) =
\Hom(L_0, S(L_1))[-1]\] 
in the sense of Seidel \cite[Section 3d, Section 12h]{se:bo}, where
$[k]$ denotes degree shift by $k$; and the functors
$\Phi: \cC_0 \to \cC_1$ between categories are required to commute
with the shift functors: $\Phi \circ S_0 = S_1 \circ \Phi$.
\end{definition}

\begin{remark} \label{fshift}
\begin{enumerate} 
\item Any Fukaya category $\Fuk(M)$ has such a shift functor obtained
  on objects by shifting the grading on the Lagrangians
  $L \mapsto L[1]$ by Seidel \cite[Section 12h]{se:bo} in the exact
  setting; the same construction holds for the extended Fukaya
  categories for monotone symplectic manifolds by taking
  $\ul{L}[1] = (L_0,\ldots, L_k[1])$ for
  $\ul{L} = (L_0,\ldots, L_k)$.
\item The functors associated to graded Lagrangian correspondences
  commute with the shift functors under the given assumption that the
  Maslov cover on any product $X_0^- \times X_1$ obtained from the
  Maslov covers on the factors as in Wehrheim-Woodward \cite[Section
  3]{quiltfloer} by
  $\Sp^N(2n_0 + 2n_1) \times_{\Sp^N(2n_0) \times \Sp^N(2n_1)}
  \Lag^N(X_0) \times \Lag^N(X_1)$,
  where $\Sp^N$ is the group of $N$-graded symplectomorphisms, with
  notation that of Seidel \cite{se:gr}.
\end{enumerate} \end{remark}

Remarks \ref{composeresult} and \ref{fshift} imply the following:

\begin{theorem}  \label{categorify}  
For any $\tau > 0 $ the maps 
$$ M \mapsto \Fuk^{\sharp}(M) , \quad 
[\ul{L}_{01}] \mapsto [\Phi(\ul{L}_{01})] $$ 
define a {\em categorification functor} 
$\Fuk^{\sharp}:
\Symp^{\sharp}_{\tau,N} \to A_\infty(N)$. 
\end{theorem} 

\begin{proof} In \cite{Ainfty} we constructed the functors
  $\Phi(\ul{L}_{01})$ and proved that if $\ul{L}_{01}, \ul{L}'_{01}$
  are related by a composition equivalence \eqref{compequiv} then the
  functors $\Phi(\ul{L}_{01}),\Phi(\ul{L}'_{01})$ are homotopic.  It
  follows that the homotopy class $[\Phi(\ul{L}_{01})]$ is independent
  of the choice of representative $\ul{L}_{01}$.
\end{proof}

\begin{definition} Two morphisms $[\Phi],[\Phi']: \cC_0 \to \cC_1$ in
  $A_\infty(N)$ are {\em shift equivalent} if after some grade shift
  the functors become homotopy equivalent, that is,
  $\Phi S_0^k \simeq \Phi'$ for some integer $k$.  Denote by
  $A_\infty(N)/S$ the category whose objects are those of
  $A_\infty(N)$ and whose morphisms spaces are shift equivalence
  classes of morphisms in $A_\infty(N)$.  Similarly, let
  $\Symp^{\sharp}_{\tau,N}/S$ be the category whose objects are those
  of $\Symp^{\sharp}_{\tau,N}$ and whose morphisms are equivalences
  classes of generalized Lagrangian branes, where two branes are
  equivalent if they are related by composition equivalences and
  grading shifts $\ul{L} \to \ul{L}[k]$.
\end{definition} 

\begin{remark}\label{rgrade} 
{\rm (Relative gradings)}
Given such an equivalence class of functors
$[\Phi] \in \Hom_{A_\infty(N)/S}(\cC_0,\cC_1)$, and objects
$L_0 \in \Obj(\cC_0), L_1 \in \Obj(\cC_1)$, the morphism space
$\Hom( \Phi(L_0),L_1)$ is independent of all choices as a {\em
  relatively} $\Z_N$-graded group.  That is, $\Hom( \Phi(L_0),L_1)$
admits a canonical decomposition into summands
$\Hom^k(\Phi(L_0), L_1), k \in \Z_N$ up to a shift in grading.  The
categorification functor
$\Fuk^{\sharp}: \Symp^{\sharp}_{\tau,N} \to A_\infty(N)$ of Theorem
\ref{categorify} induces a relatively graded version
  \[ \Fuk^{\sharp}/S:\ \  \Symp^{\sharp}_{\tau,N}/S \to A_\infty(N)/S .\]
  In particular, given a morphism
  $[L] \in \Hom_{\Symp^{\sharp}_{\tau,N}/S}(M_0,M_1)$, and objects
  $L_0 \in \Obj( \GFuk(M_0)), L_0 \in \Obj( \GFuk(M_1))$, the morphism
  space $\Hom(\Phi(\ul{L})(L_0), L_1)$ is a $\Z_N$-relatively-graded
  group.
\end{remark}

\subsection{Floer field theory}

We can now prove the main result of the paper:

\begin{theorem} \label{fft}
{\rm (Floer field theory)} For any coprime positive integers $r,d$,
the maps for surfaces $X$ given by 
$$X \mapsto \Phi(X):= \Fuk^{\sharp}(M(X))$$ 
and for elementary morphisms $[Y]$ given by 
$$[Y] \mapsto \Phi([Y]) := [\Phi(L(Y))] $$ 
extend to a connected field theory 
$$\Phi: \Cob^0_{2+1} \to A_\infty(4)/S$$
with values in the category \ainfty of (small $\Z_4$-graded \ainfty
categories with shift, homotopy classes of \ainfty functors up to
grading shifts).
\end{theorem}

\begin{proof} The desired functor is obtained by composing the
  functors $\Phi$ of Theorem \ref{sympval} (assigning a generalized
  Lagrangian correspondence to any cobordism) with the functor
  $\GFuk/S$ of Remark \ref{rgrade} with $ \tau = 1/2r$ and
  $N = 4$ (assigning a functor to any such correspondence, modulo
  grading shifts).
\end{proof} 

\begin{remark} 
\begin{enumerate} 
\item {\rm (Group-valued invariants for bordisms)} Let
  $L_\pm \subset M(X_\pm)$ be simply-connected compact oriented spin
  Lagrangian submanifolds.  For a bordism $Y$ from $X_-$ to $X_+$ we
  denote by $ [\ul{L}(Y)] $ the corresponding equivalence class of
  generalized Lagrangian correspondences from $M(X_-)$ to $M(X_+)$.
  The quilted Floer cohomology group
$$ HF(Y;L_-,L_+) := H(\Hom( \Phi([\ul{L}(Y)])L_-, L_+)) = HF(L_-,
\ul{L}(Y) , L_+) $$
is a {\em topological invariant} of the bordism $Y$ up to isomorphism.
These Floer cohomology groups are relatively $\Z_4$-graded, since the
minimal Chern number of the moduli space $M(X)$ of bundles is even by
Proposition \ref{admiss}.
\item {\rm (Excision property)} The functors $\Phi(Y)$ satisfy a
  gluing property for gluing: if $Y_{02} = Y_{01} \circ Y_{12}$ then
$$ \Phi(Y_{02}) = \Phi(Y_{12}) \circ \Phi(Y_{01}) .$$
This relation can be regarded as a generalization of Floer's excision
property \cite{br:fl}. If $X_1$ has genus zero then $\Phi(Y) $ is the
trivial functor, while if $X_1$ has genus one then
\begin{equation} \label{tensor} H(\Hom(\Phi(Y_{02})L_-,L_+)) \cong
H(\Hom(\Phi(Y_{01})L_-,\pt) \otimes \Hom(\Phi(Y_{12})\pt,L_+))
.\end{equation}
Equation \eqref{tensor} follows immediately from Theorem \ref{MP}
since in this case $M(X_1)$ is empty resp. a point and so the Floer
complex is the product of Floer complexes
$\Hom(\Phi(Y_{01})L_-,\pt), \Hom(\Phi(Y_{12})\pt,L_+))$ for
$(\Phi(Y_{01})L_-,\pt)$ and $(\Phi(Y_{12}),L_+)$.
\end{enumerate} 
\end{remark} 

\subsection{Invariants for three-manifolds with circle-valued Morse functions} 

In this section we define invariants for three-manifolds which admit
circle-valued Morse functions.  Let $Y$ be a compact connected
oriented three-manifold containing a non-separating embedded compact
connected oriented surface $X$.  By replacing $X$ with two copies of
itself we obtain a bordism $\ol{Y}$ from $X$ to itself.  Any adapted
Morse function on $\ol{Y}$ gives rise to a circle-valued Morse
function on $Y$.

\begin{definition} The {\em cyclic Floer homology} of $(Y,X)$ is the quilted
Floer cohomology $ HF(Y,X) := HF( \ul{L}(\ol{Y}) , \Delta_{M(X)} ) .$
\end{definition} 

\begin{remark} \label{gaykirby}
\begin{enumerate} 
\item The results of Gay-Kirby \cite[Theorem 1.2]{gk:indef} imply that
  any two decompositions of $Y$ into elementary bordisms are related
  by Cerf moves.  It follows that the equivalence class of the
  generalized cyclic correspondence $ (\Delta_{M(X)},\ul{L}(X))$
  depends only on the homotopy class $[f: Y \to S^1]$ of the
  circle-valued Morse function.  Furthermore, any homotopy class
  $[f: Y \to S^1]$ of maps to the circle inducing a surjection of
  fundamental groups has a representative $f:Y \to S^1$ that is a
  Morse function with connected fibers:
$$ \# \pi_0(f^{-1}(b)) = 1, \quad \forall b \in S^1.$$
Such a Morse function gives rise to a presentation of $Y$ as a cyclic
composition of elementary bordisms $Y_1,\ldots, Y_m$.
\item The homotopy class of the circle-valued Morse functions is the
  first cohomology class on $Y$, via the identification of $S^1$ with
  the first Eilenberg-Maclane space.  Thus the isomorphism class of
  $HF(Y,X)$ depends only on the choice of cohomology class $[X]^\dual
  \in H^1(Y,\Z)$ dual to $[X]$.
\item 
 If $Y = S^1 \times X$ is the trivial fiber bundle with fiber $X$ then
$$HF(Y,X) =  HF(\Delta_{M(X)})
 \cong QH(M(X))$$ 
is the quantum cohomology of the moduli space $M(X)$.  More generally,
for any fiber bundle with monodromy $\psi: X \to X$ we may take for
$\ul{L}(Y)$ the graph of a symplectomorphism $(\psi^{-1})^* : M(X) \to
M(X)$.  In this case 
$$HF(Y,X) = HF( \graph(\psi^{-1})^* ) = HF^{\on{inst}}(Y) $$
is equivalent to the periodic instanton Floer cohomology
$HF^{\on{inst}}(Y)$ by the proof of the Atiyah-Floer conjecture for
fibrations by Dostoglou-Salamon \cite{ds:sd}.  Partial results towards
a correspondence with instanton Floer homology in the circle-valued
Morse case are given in Duncan \cite{du:qaf} and Lipyanskiy
\cite{lip:gu}.
\end{enumerate}
\end{remark} 

\subsection{Invariants for closed $3$-manifolds}

The invariant associated to a bundle over a closed three-manifold,
considered as a bordism from the empty surface to itself is trivial.
Indeed, such a three-manifold admits a Cerf decomposition where one of
the surfaces is a sphere and the moduli space of bundles of coprime
rank and degree is empty in genus zero.  The following device
suggested by Kronheimer and Mrowka \cite{km:kh} gives non-trivial
invariants of closed three-manifolds, possibly with trivial first
Betti number.

\begin{definition} \label{torussummed}
\begin{enumerate} 
\item {\rm (Torus-summed three-manifolds)} Given a closed oriented
  three-manifold $Y$ let $\ol{Y} := Y {\sharp} ([-1,1] \times T^2)$
  denote the {\em torus-summed} bordism obtained as connected sum of
  $Y$ with a product bordism $[-1,1] \times T^2$ between tori $T^2$.
\item {\rm (Torus-summed Floer homology) }
The {\em torus-summed Lagrangian Floer homology} for rank $r$ and
coprime degree $d$ of a closed three-manifold $Y$ is
$$ \ol{HF}(Y) := H(\Hom(\Phi(\ol{Y}) \pt, \pt)) = HF (
L(Y_1),\ldots,L(Y_m)) $$
for some choice of Cerf decomposition $\ol{Y} = Y_1 \cup_{X_1} \ldots
\cup_{X_{r-1}} Y_m$.
\end{enumerate} 
\end{definition}

\begin{example} {\rm (Torus-summed Lagrangian Floer homology for connect sums of
  $S^2 \times S^1$)} Suppose that $r = 2$, and for some positive
  integer $n$ the three-manifold 
$$Y  = \underbrace{(S^2 \times S^1) \sharp \ldots \sharp (S^2 \times S^1)}_n $$ 
is the connected sum of $n$ copies of $S^2 \times S^1$.  Then
$\ol{HF}(Y) = H(S^3)^n$.  Indeed, consider a Morse function on $Y$
that splits $Y {\sharp} ([-1,1] \times T^2)$ into compression bodies
$Y_\pm$, each with $n$ critical points of index one resp. two.  Each
compression body $Y_\pm$ is a bordism between a surface $X$ of genus
$n+1$ and a surface of genus $1$.  The Lagrangians
$$L(Y_-) = L(Y_+) \cong SU(2)^n \cong (S^3)^n$$ 
are identical and given in the holonomy description by
$$ L(Y_\pm) = \{ a_1 = \ldots = a_n = 1 \} \subset \left\{
\prod_{i=1}^{n+1} [a_i,b_i] = -1 \right\} .$$
By definition $\ol{HF}(Y) = HF(L(Y_-),L(Y_+)) .$ As in Biran and
Cornea \cite[Proposition 6.1.4]{bc:ql}, if $L$ is a Lagrangian product
of spheres, then $HF(L,L)$ is either $0$ or isomorphic to the ordinary
cohomology $H(L)$; this is a consequence of the spectral sequence
computing Floer cohomology starting from the singular cohomology, as
in Oh \cite{oh:fc}, \cite{buh:mul}.  By Albers \cite[Corollary
  2.11]{alb:lag}, $HF(L,L;\Z_2)$ is non-trivial if the class
$[L]_{\Z_2} \in H(M(X);\Z_2)$ of $L$ is non-zero.  The class of
$L(Y_\pm) \cong (S^3)^n$ is non-zero in $H(M(X);\Z_2)$.  Indeed, $L$
intersects the Lagrangian $L' \cong (S^3)^n$, given by $b_1 = \ldots =
b_n = 1$ transversally in a single point, namely the unique (up to
conjugation) pair $a_{n+1},b_{n+1}$ with $[a_{n+1},b_{n+1}] = 1$.
Hence $HF(L,L;\Z_2)$ is non-zero.  This implies that $HF(L,L)$ is
non-zero by the universal coefficient theorem.  Putting everything
together implies the claim.
\end{example}  

\begin{remark}  
\begin{enumerate}
\item {\rm (Four-manifold invariants?)}  Naturally one expects these
  invariants to extend to invariants of four-manifolds, fitting into a
  bundle field theory in $2+1+1$ dimensions: a bifunctor from a
  bicategory of ($2$-manifolds, $3$-bordisms, $4$-bordisms) with
  bundles to an ``\ainfty bicategory'' of (\ainfty categories, \ainfty
  functors, \ainfty natural transformations) via rank $r$ gauge
  theory.  In particular, this theory should associate to a bundle
  bordism $R$ between bordisms $Y_\pm$ a natural transformation of
  \ainfty functors
$ \Pi(R): \Phi(Y_-) \to \Phi(Y_+) .$
\item {\rm (Surgery exact triangles)} The surgery exact triangles for
  the theory (in the rank two case) are a consequence of a fibered
  generalization of a triangle for Dehn twists by Seidel, proved in
  \cite{wo:ex}.
\end{enumerate}
\end{remark} 


\def\cprime{$'$} \def\cprime{$'$} \def\cprime{$'$} \def\cprime{$'$}
  \def\cprime{$'$} \def\cprime{$'$}
  \def\polhk#1{\setbox0=\hbox{#1}{\ooalign{\hidewidth
  \lower1.5ex\hbox{`}\hidewidth\crcr\unhbox0}}} \def\cprime{$'$}
  \def\cprime{$'$}

\end{document}